\documentclass[a4paper,leqno,10pt]{amsart}

\usepackage{epsf,graphicx,epsfig}
\usepackage{amscd}
\usepackage{amsmath,latexsym,amssymb,amsthm}
\usepackage[nospace,noadjust]{cite}
\usepackage{textcomp}
\usepackage{setspace,cite}
\usepackage{lscape,fancyhdr,fancybox}
\usepackage{stmaryrd}
\usepackage[all,cmtip]{xy}
\usepackage{tikz}
\usepackage{cancel}
\usetikzlibrary{shapes,arrows,decorations.markings}
\usepackage{graphicx}

\usepackage{color}
\usepackage{url}
\usepackage{enumerate}
\usepackage[mathscr]{euscript}

  \theoremstyle{plain}

\swapnumbers
    \newtheorem{thm}{Theorem}[section]
    \newtheorem{prop}[thm]{Proposition}
     
   \newtheorem{lemma}[thm]{Lemma}
    \newtheorem{corollary}[thm]{Corollary}
    
    \newtheorem{subsec}[thm]{}
\theoremstyle{definition}
    \newtheorem{defn}[thm]{Definition}
        \newtheorem{remark}[thm]{Remark}
    \newtheorem{exam}[thm]{Example}

\theoremstyle{remark}

\title{}
\author{}
\date{}
\usepackage{amssymb}

\usepackage{hyperref}
\hypersetup{
	colorlinks,
	citecolor=blue,
	filecolor=black,
	linkcolor=blue,
	urlcolor=black
}

\begin{document}
\title{Generalized Rota-Baxter systems}

\author{Apurba Das}
\address{Department of Mathematics and Statistics,
Indian Institute of Technology, Kanpur 208016, Uttar Pradesh, India.}
\email{apurbadas348@gmail.com}



\maketitle

\begin{abstract}
Rota-Baxter systems of T. Brzezi\'{n}ski are a generalization of Rota-Baxter operators that are related to dendriform structures, associative Yang-Baxter pairs and covariant bialgebras. In this paper, we consider Rota-Baxter systems in the presence of bimodule, which we call generalized Rota-Baxter systems. We define a graded Lie algebra whose Maurer-Cartan elements are generalized Rota-Baxter systems. This allows us to define a cohomology theory for a generalized Rota-Baxter system. Formal one-parameter deformations of generalized Rota-Baxter systems are discussed from cohomological points of view. We further study Rota-Baxter systems, associative Yang-Baxter pairs, covariant bialgebras and introduce generalized averaging systems that are related to associative dialgebras. Next, we define generalized Rota-Baxter systems in the homotopy context and find relations with homotopy dendriform algebras. The paper ends by considering commuting Rota-Baxter systems and their relation with quadri-algebras.\\
\end{abstract}

\noindent

\thispagestyle{empty}

\tableofcontents

\vspace{0.2cm}

\section{Introduction}
The notion of Rota-Baxter operators was first introduced by Baxter \cite{baxter} in the context of differential operators on commutative Banach algebras and further studied by Rota \cite{rota} in connections with probability and combinatorics. There are important applications of Rota-Baxter operators in the Connes-Kreimer's algebraic approach of the renormalization in quantum field theory \cite{conn}. In \cite{aguiar-pre} Aguiar showed that Rota-Baxter operators produce dendriform structures. Aguiar \cite{aguiar,aguiar-pre-lie} also introduced an associative analogue of Yang-Baxter equation and extensively study infinitesimal bialgebras that were first appeared in a paper by Joni and Rota \cite{joni-rota}. An associative Yang-Baxter solution gives rise to a Rota-Baxter operator and an infinitesimal bialgebra. In \cite{uchino} Uchino introduced a generalization of Rota-Baxter operators that can also be considered as an associative analogue of Poisson structures. See \cite{guo-book} for more details on Rota-Baxter operators. Recently, deformations of associative Rota-Baxter operators and their governing cohomology has been studied in \cite{das2}.

\medskip

In \cite{brz} Brzezi\'{n}ski introduced Rota-Baxter systems as a generalization of a Rota-Baxter operator. In a Rota-Baxter system, two operators are acting on the algebra and satisfying some Rota-Baxter type identities. More precisely, a pair $(R, S)$ of linear maps on an algebra $A$ is said to be a {\bf Rota-Baxter system} if they satisfy
\begin{align*}
R(a) R(b) =~& R ( R(a) b + a S(b)),\\
S(a) S(b) =~& S ( R(a) b + a S(b)),
\end{align*}
for all $a, b \in A$. He also introduced associative Yang-Baxter pairs (as a generalization of Yang-Baxter solutions), covariant bialgebras (as a generalization of infinitesimal bialgebras) and find various relations among them that generalize the results of Aguiar. See \cite{qiu-chen} for the free construction of Rota-Baxter system and related structures.

\medskip

Our main objective in this paper is Brzezi\'{n}ski's Rota-Baxter systems in the presence of $A$-bimodule $M$. Motivated by Uchino's \cite{uchino} terminology of generalized Rota-Baxter operator, we call Rota-Baxter systems in the presence of bimodule as generalized Rota-Baxter systems. They were called $\mathcal{O}$-operator system in \cite{ma-et}. Thus, Rota-Baxter systems on an associative algebra $A$ are generalized Rota-Baxter systems on the adjoint $A$-bimodule $A$. A generalized Rota-Baxter system induces a dendriform structure on $M$, hence an associative structure on $M$. We give some new characterizations of (generalized) Rota-Baxter systems. Motivated from gauge transformations \cite{sev-wein} and reductions \cite{mars-ratiu} of Poisson structures, we define some new constructions of generalized Rota-Baxter systems.

\medskip

Next, given an associative algebra $A$ and an $A$-bimodule $M$, we construct a graded Lie algebra whose Maurer-Cartan elements are given by generalized Rota-Baxter systems. This characterization of a Rota-Baxter system allows us to define a cohomology theory for a generalized Rota-Baxter system. We further show that this cohomology is isomorphic to the Hochschild cohomology of the induced associative algebra $M$ with coefficients in a suitable bimodule structure on the direct sum $A \oplus A$. We also find a morphism from the cohomology of a generalized Rota-Baxter system to the cohomology of the induced dendriform algebra.

\medskip

The classical deformation theory of Gerstenhaber \cite{gers} for associative algebras has been extended to various other algebraic structures including Lie algebras \cite{nij-ric} and Leibniz algebras \cite{balav}. More generally, Balavoine \cite{bala} considered deformations of algebras over binary quadratic operads. Deformations of morphisms have also been considered in many articles \cite{gers-sch,yael}. In \cite{das2} the author study deformations of Rota-Baxter operators from cohomological perspectives. Here we consider deformations of generalized Rota-Baxter systems. We show that the linear term in a formal deformation of a given generalized Rota-Baxter system is a $1$-cocycle in the above-defined cohomology of the generalized Rota-Baxter system. Finally, given a finite order deformation, we construct a $2$-cocycle in the cohomology. The corresponding cohomology class is called the obstruction class. When the obstruction class vanishes, the given deformation extends to next order.

\medskip

We also study some further results on Rota-Baxter systems, associative Yang-Baxter pairs and covariant bialgebras that are introduced by Brzezi\'{n}ski \cite{brz}. We show that a pair $(r,s)$ of skew-symmetric elements of $A^{\otimes 2}$ is an associative Yang-Baxter pair if and only if $(r^\sharp, s^\sharp)$ is a generalized Rota-Baxter system on the coadjoint $A$-bimodule $A^*$. Following the perturbations of quasi-Hopf algebra  \`{a} la Drinfel'd, we study perturbations of the coproduct in a covariant bialgebra. We find a necessary and sufficient condition for the perturbed object to be a new covariant bialgebra. A suitably compatible covariant bialgebra also induces a pre-Lie algebra structure that generalizes a result of Aguiar \cite{aguiar-pre-lie} for infinitesimal bialgebras. Finally, we introduce a notion of generalized averaging system as a generalization of an averaging operator \cite{pei-guo,rey}. A generalized averaging system induces a dialgebra structure in the sense of Loday \cite{loday}. We observed that a generalized averaging system can be seen as a particular case of generalized Rota-Baxter system. Therefore, all cohomological results and deformations of generalized Rota-Baxter systems can be applied to generalized averaging systems.

\medskip

The notion of strongly homotopy associative algebras or $A_\infty$-algebras were introduced by Stasheff in \cite{sta}. We introduce generalized Rota-Baxter systems on a bimodule over an $A_\infty$-algebra. We show that a generalized Rota-Baxter system induces a $Dend_\infty$-algebra structure on the underlying bimodule space. This generalizes a result of Brzezi\'{n}ski in the homotopy context.

\medskip

Finally, we consider Rota-Baxter systems on a dendriform algebra and show how they induce quadri-algebra structures. As a consequence, we obtain a quadri-algebra structure from two commuting Rota-Baxter systems on an associative algebra. The homotopy analogue of these results is also described.

\medskip

The paper is organized as follows. In Section \ref{sec-2}, we introduce generalized Rota-Baxter systems and give some characterizations and new constructions. In Section \ref{sec-3}, we mainly study cohomology of a generalized Rota-Baxter system and find their relationship with the dendriform algebra cohomology. Deformations of generalized Rota-Baxter systems are considered in Section \ref{sec-4}. Applications and some further properties of Rota-Baxter systems, associative Yang-Baxter pairs, covariant bialgebras and generalized averaging systems are given in Section \ref{sec-5}. In Section \ref{sec-6}, we define generalized Rota-Baxter systems in homotopy context and find their relations with $Dend_\infty$-algebras. Finally, in Section \ref{sec-7}, we consider commuting Rota-Baxter systems and obtain quadri-algebras.

\medskip

All vector spaces, linear maps and tensor products are over a field $\mathbb{K}$ of characteristic zero. The elements of the associative algebra $A$ are usually denoted by $a,b,c, \ldots$ and the elements of $M$ are denoted by $u, v, w, \ldots$.

\section{Generalized Rota-Baxter systems}\label{sec-2}

In this section, we introduce generalized Rota-Baxter systems and some of their characterizations. Finally, we define gauge transformations and reductions of generalized Rota-Baxter systems.

Let $(A, \mu)$ be an associative algebra. We denote the multiplication $\mu (a,b)$ simply by $ab$ and denote the algebra simply by $A$. An $A$-bimodule is a vector space $M$ together with a left $A$-action $l: A \otimes M \rightarrow M,~(a,u) \mapsto a \cdot u$ and a right $A$-action $r: M \otimes A \rightarrow M,~(u,a) \mapsto u \cdot a$ satisfying the following compatibilities
\begin{align*}
(ab) \cdot u = a \cdot (b \cdot u), \qquad (a \cdot u) \cdot b = a \cdot (u \cdot b), \qquad (u \cdot a) \cdot b = u \cdot (ab),
\end{align*}
for $a, b \in A,~ u \in M$. An $A$-bimodule is denoted by the triple $(M,l,r)$  or simply by $M$. 

It follows that any associative algebra $A$ is an $A$-bimodule (called adjoint bimodule) in which the left and right actions are given by the algebra multiplication map. Moreover, the dual space $A^*$ also carries an $A$-bimodule structure (called coadjoint bimodule) with left and right actions given by
\begin{align*}
(a \cdot f ) (b) = f (ba )   \quad \text{ and } \quad (f \cdot a) (b) = f (ab), ~ \text{ for } a, b \in A, f \in A^*.
\end{align*}

Let $A$ be an associative algebra and $M$ an $A$-bimodule. 
\begin{defn}
A {\bf generalized Rota-Baxter system} on $M$ over the algebra $A$ consists of two linear maps $R, S : M \rightarrow A$ satisfying the following identities: for all $u,v \in M,$
\begin{align}
R(u) R(v) =~& R ( R(u) \cdot  v + u \cdot S(v)), \label{rb-system1}\\
S(u) S(v) =~& S ( R(u) \cdot v + u \cdot S(v)). \label{rb-system2}
\end{align}
\end{defn}

\begin{exam}
A Rota-Baxter system is a generalized Rota-Baxter system on the adjoint bimodule $A$. 
\end{exam}

\begin{exam}\label{tw-rota}
Let $A$ be an associative algebra, $M$ be an $A$-bimodule and $\sigma : M \rightarrow A$ be an algebra map. A linear map $R: M \rightarrow A$ is said to be a $\sigma$-twisted generalized Rota-Baxter operator if it satisfies
\begin{align*}
R(u) R(v) = R (R(u) \cdot v + u \cdot (\sigma \circ R)(v)),~ \text{ for } u, v \in M.
\end{align*}
Let $A = \mathbb{K}[x]$ be the polynomial ring in one variable and define an algebra map $\sigma$ by $\sigma (x^n) = q^n x^n$ ($q \neq 1$). Then the Jackson integral $J: A \rightarrow A$ given by
\begin{align*}
J(x^n) = \frac{1-q}{1 - q^{n+1}}~ x^{n+1}
\end{align*}
is a $\sigma$-twisted (generalized) Rota-Baxter operator. It is easy to see that if $R$ is a $\sigma$-twisted generalized Rota-Baxter operator then $(R, S = \sigma \circ R)$ is a generalized Rota-Baxter system.
\end{exam}

Some other examples of (generalized) Rota-Baxter systems are given in Subsection \ref{subsec-aybp}.

\medskip

Note that given an associative algebra $A$ and an $A$-bimodule $M$, the sum $A \oplus A \oplus M$ carries an associative algebra structure with product
\begin{align*}
(a_1, a_2, u) ( b_1, b_2, v) = ( a_1 b_1, a_2 b_2, a_1 \cdot  v + u \cdot b_2 ).
\end{align*}
This is exactly the semi-direct product if we consider the direct sum associative algebra structure on $A \oplus A$ and the bimodule structure on $M$ given by $(a_1, a_2) \cdot m = a_1 \cdot m$ and $m \cdot (a_1, a_2) = m \cdot a_2$.

\begin{prop}
A pair $( R, S)$ of maps from $M$ to $A$ is a generalized Rota-Baxter system if and only if the pair of maps
\begin{align*}
\widehat{R}: A \oplus A \oplus M \rightarrow A \oplus A \oplus M, ~(a_1, a_2, u) \mapsto ( R(u), 0, 0) \\
\widehat{S}  : A \oplus A \oplus M \rightarrow A \oplus A \oplus M, ~ (a_1, a_2, u ) \mapsto ( 0, S(u), 0)
\end{align*} 
is a Rota-Baxter system on the algebra $A \oplus A \oplus M$.
\end{prop}

\begin{proof}
We have $\widehat{R} (a_1, a_2, u)  \widehat{R}(b_1, b_2, v) = ( R(u) R(v), 0,0)$ and
\begin{align*}
\widehat{R} \big(   ( \widehat{R} (a_1, a_2, u) ) (b_1, b_2, v) + (a_1, a_2, u) ( \widehat{S}(b_1, b_2, v)) \big) = \big( R (R(u) v + u S(v))  ,0,0 \big).
\end{align*}
Similarly, $\widehat{S} (a_1, a_2, u) \widehat{S}(b_1, b_2, v) = ( 0, S(u) S(v), 0)$ and 
\begin{align*}
\widehat{S} \big(  (  \widehat{R} (a_1, a_2, u) ) (b_1, b_2, v) + (a_1, a_2, u) ( \widehat{S}(b_1, b_2, v)) \big) = \big( 0,  S (R(u) v + u S(v))  ,0\big).
\end{align*}
Hence $(R,S)$ is a generalized Rota-Baxter system if and only if $(\widehat{R}, \widehat{S})$ is a Rota-Baxter system.
\end{proof}

\begin{prop}\label{rbs-graph}
A pair $(R, S)$ is a generalized Rota-Baxter system if and only if the graph
\begin{align*}
\mathrm{Gr}((R, S)) := \{ (Ru, Su, u) | ~u \in M \} \subset A \oplus A \oplus M
\end{align*}
is a subalgebra.
\end{prop}

\begin{proof}
For any $u, v \in M$, we have
\begin{align*}
(R(u), S(u), u) (R(v), S(v), v) = ( R(u) R(v) , S(u) S(v) , R(u) \cdot v + u \cdot S(v))
\end{align*}
Hence $\mathrm{Gr}((R, S))$ is a subalgebra if and only if $(R, S)$ is a generalized Rota-Baxter system.
\end{proof}

Nijenhuis operators on associative algebras \cite{grab} are associative analog of classical Nijenhuis operator on Lie algebras. More precisely, a linear map $N : A \rightarrow A$ on an algebra $A$ is said to be a Nijenhuis operator if it satisfies
\begin{align*}
N(a) N(b) = N \big( N(a) b + a N(b) - N (ab) \big), \text{ for } a, b \in A.
\end{align*}

The following result relates to generalized Rota-Baxter systems and Nijenhuis operators.

\begin{prop}
A pair $(R,S)$ is a generalized Rota-Baxter system if and only if 
\begin{align*}
N_{(R,S)} = \begin{pmatrix} 
0 & 0 & R \\
0 & 0 & S \\
0 & 0 & 0 
\end{pmatrix} : A \oplus A \oplus M \rightarrow A \oplus A \oplus M
\end{align*}
is a Nijenhuis operator on the algebra $A \oplus A \oplus M$.
\end{prop}

\begin{proof}
By a simple calculation, we have
\begin{align}\label{nij-1}
N_{(R,S)} (a_1, a_2, u)  N_{(R,S)} (b_1, b_2, v) = ( R(u) R(v), S(u) S(v), 0)
\end{align}
and 
\begin{align}\label{nij-2}
&N_{(R,S)} \bigg( N_{(R,S)} (a_1, a_2, u) (b_1, b_2, v) + (a_1, a_2, u) N_{(R,S)} (b_1, b_2, v) - N_{(R,S)} ((a_1, a_2, u) (b_1, b_2, v))     \bigg) \nonumber \\
&= (R (R(u) \cdot v + u \cdot S(v)) , S (R(u) \cdot v + u \cdot S(v)), 0 ).
\end{align}
It follows from (\ref{nij-1}) and (\ref{nij-2}) that $N_{(R,S)}$ is a Nijenhuis tensor if and only if $(R,S)$ is a generalized Rota-Baxter system.
\end{proof}

Let $A$ be an associative algebra and $M$ an $A$-bimodule. We assume that $\mathrm{dim } A = \mathrm{dim } M$. A pair $(\theta_0, \theta_1)$ of invertible linear maps from $A$ to $M$ is said to be an invertible $1$-cocycle system if they satisfy
\begin{align*}
\theta_0 (xy ) =~& x \cdot \theta_0 (y ) + \theta_0 (x)   \cdot  (\theta_1^{-1} \circ \theta_0 (y)),\\
\theta_1 (xy ) =~& (\theta_0^{-1} \circ \theta_1 (x) ) \cdot \theta_1 (y) + \theta_1(x) \cdot y,
\end{align*}
for all $x, y \in A$. It follows from the above definition that $(\theta, \theta)$ is an invertible $1$-cocycle system if and only if $\theta : A \rightarrow M$ is an invertible $1$-cocycle.

\begin{prop}
Let $A$ be an associative algebra and $M$ an $A$-bimodule ($\mathrm{dim } A = \mathrm{dim } M$). A pair $(R, S)$ of invertible linear maps from $M$ to $A$ is a generalized Rota-Baxter system if and only if $(R^{-1}, S^{-1})$ is an invertible $1$-cocycle system.
\end{prop}

\begin{proof}
It is easy to see that the identity (\ref{rb-system1}) is equivalent to (by taking $R(u) = x, R(v) = y$)
\begin{align*}
R^{-1} (xy) = x \cdot R^{-1} (y) + R^{-1}(x) \cdot ( (S^{-1})^{-1} \circ R^{-1} (y)).
\end{align*}
Similarly, the identity (\ref{rb-system2})
 is equivalent to (by taking $S(u) = x', s(v) = y'$)
 \begin{align*}
 S^{-1} (x' y') = ((R^{-1})^{-1} \circ S^{-1} (x')) \cdot S^{-1} (y') + S^{-1} (x') \cdot y'.
 \end{align*}
 It follows that $(R, S)$ is a generalized Rota-Baxter system if and only if $(R^{-1}, S^{-1})$ is an invertible $1$-cocycle system.
 \end{proof}
 
 \medskip

\begin{defn} \cite{loday}
A {\bf dendriform algebra} is a triple $(D, \prec, \succ)$ consisting of a vector space $D$ together with bilinear operations $\prec, \succ : D \otimes D \rightarrow D$ satisfying
\begin{align*}
 (a \prec b) \prec c =~& a \prec (b \prec c + b \succ c),\\
 (a \succ b) \prec c =~&  a \succ (b \prec c),\\
 (a \prec b + a \succ b) \succ c =~& a \succ (b \succ c), ~~~ \text{ for all } a, b , c \in D.
\end{align*} 
\end{defn}

It follows that in a dendriform algebra $(D, \prec, \succ)$, the sum operation 
\begin{align*}
a * b =  a \prec b + a \succ b ~~\text{ is associative.}
\end{align*}

Let $(D, \prec, \succ)$ and $(D', \prec', \succ')$ be two dendriform algebras. A morphism between them is given by a linear map $f : D \rightarrow D'$ satisfying $f (a \prec b ) = f (a) \prec' f(b)$ and $f (a \succ b ) = f (a) \succ' f(b)$, for $a, b \in D$.

The following result relates to generalized Rota-Baxter systems and dendriform algebras \cite{brz}.

\begin{prop}\label{RBS-dend}
Let $(R,S)$ be a generalized Rota-Baxter system on $M$ over the algebra $A$. Then $M$ carries a dendriform algebra structure with 
\begin{align*}
u \prec v = u \cdot S(v) ~~~~ \text{ and } ~~~~ u \succ v = R(u) \cdot v, ~ \text{ for } u,v \in M.
\end{align*}
Consequently, there is an associative algebra structure on $M$ given by $u * v = u \prec v + u \succ v$ for which both $R$ and $S$ are morphisms of associative algebras.
\end{prop}

Next, we define morphisms between generalized Rota-Baxter systems. Let $(R, S)$ be a generalized Rota-Baxter system on an $A$-bimodule $M$ and $(R', S')$ be a generalized Rota-Baxter system on an $A'$-bimodule $M'$.

\begin{defn}\label{grbs-mor}
A morphism of generalized Rota-Baxter systems from $(R, S)$ to $(R', S')$ consists of a triple $(\phi, \varphi, \psi)$ of  algebra maps $\phi, \varphi : A \rightarrow A'$ and a linear map $\psi : M \rightarrow M'$ satisfying
\begin{itemize}
\item[(i)] $R' \circ \psi = \phi \circ R$ ~~~ and ~~~ $S' \circ \psi = \varphi \circ S$,
\item[(ii)] $ \psi (a \cdot u)  = \phi (a)\cdot \psi (u) $ ~~~ and ~~~ $\psi (u \cdot a) = \psi (u) \cdot \varphi(a)$.
\end{itemize}
It is called an isomorphism if $\phi, \varphi$ and $\psi$ are all linear isomorphisms.
\end{defn} 

\begin{prop}
Let $(\phi, \varphi, \psi)$ be a morphism of generalized Rota-Baxter systems from $(R, S)$ to $(R', S')$ as in the above definition. Then $\psi : M \rightarrow M'$ is a morphism between induced dendriform structures.
\end{prop}

\begin{proof}
For any $u,v \in M$, we have
\begin{align*}
\psi ( u \prec v) = \psi ( u \cdot S(v)) = \psi (u) \cdot  \varphi S(v) = \psi (u)  \cdot S' ( \psi(v)) = \psi (u) \prec' \psi (v).
\end{align*}
Similarly,
\begin{align*}
\psi ( u \succ v) = \psi ( R(u) \cdot v) = \phi R(u)  \cdot \psi (v) = R' ( \psi(u)) \cdot \psi (v) = \psi (u) \succ' \psi (v).
\end{align*}
Hence the proof.
\end{proof}

\begin{defn}
A {\bf (left) pre-Lie algebra} is a vector space $L$ together with a bilinear operation $\diamond : L \otimes L \rightarrow L$ satisfying
\begin{align*}
( a \diamond b ) \diamond c - a \diamond ( b \diamond c ) = ( b \diamond a ) \diamond c - b \diamond ( a \diamond c ), \text{ for } a, b, c \in L. 
\end{align*}
\end{defn}

The following result is well-known \cite{aguiar-pre-lie,aguiar-pre}.
\begin{prop}
(i) Let $(L, \diamond)$ be a pre-Lie algebra. Then the vector space $L$ is equipped with the bracket $[a,b] = a \diamond b - b \diamond a$ is a Lie algebra, called the subadjacent Lie algebra.

(ii) Let $(D, \prec, \succ)$ be a dendriform algebra. Then $(D, \diamond)$ is a pre-Lie algebra, where $a \diamond b = a \succ b - a \prec b$, for $a, b \in D$.
\end{prop}

Thus, it follows that if $(R,S)$ is a generalized Rota-Baxter system on $M$ over the algebra $A$, then there is a pre-Lie structure on $M$ given by $u \diamond v = R(u) \cdot v - u \cdot S(v)$, for $u,v\in M$.

\medskip

\subsection{Gauge transformations} Gauge transformations of Poisson structures by suitable closed $2$-forms was introduced by \v{S}evera and Weinstein \cite{sev-wein}. Motivated from the fact that generalized Rota-Baxter operators (more generally, generalized Rota-Baxter systems) are the associative analogue of Poisson structures, we define here gauge transformations of generalized Rota-Baxter systems.

Let $L \subset A \oplus A \oplus M$ be a subalgebra of the semi-direct product. For any linear map $B : A \oplus A \rightarrow M$, we define a subspace
\begin{align*}
\tau_B (L ) := \{ ( a_1, a_2 , u + B (a_1, a_2 ) ) |~ (a_1, a_2, u ) \in L \}.
\end{align*} 

Then we have the following.

\begin{prop}\label{gauge-prop}
The subspace $\tau_B (L) \subset A \oplus A \oplus M$ is a subalgebra of the semi-direct product if and only if $B$ is a $1$-cocycle in the Hochschild cohomology of the algebra $A \oplus A$ with coefficient in $M$. 
\end{prop}

\begin{proof}
For any $(a_1, a_2, u), (b_1, b_2, v) \in L$, we have
\begin{align*}
&(a_1, a_2, v + B (a_1, a_2)) (b_1, b_2, n + B (b_1, b_2)) \\
&= (a_1 b_1, a_2 b_2, a_1 \cdot v + a_1 \cdot B (b_1, b_2) + u \cdot b_2 + B(a_1, a_2) \cdot b_2 ).
\end{align*}
This shows that $\tau_B (L)$ is a subalgebra if and only if $a_1 \cdot B (b_1, b_2) + B(a_1, a_2 ) \cdot b_2 = B (a_1 b_1, a_2 b_2),$ or equivalently, 
\begin{align*}
(a_1, a_2) \cdot B (b_1, b_2) - B ( (a_1, a_2) (b_1, b_2)) + B (a_1, a_2) \cdot (b_1, b_2) = 0.
\end{align*}
In other words, $B$ is a $1$-cocycle.
\end{proof}

Next, let $(R,S)$ be a generalized Rota-Baxter system on $M$ over the algebra $A$. Consider the subalgebra $\mathrm{Gr}((R,S)) = \{ ( Ru, Su, u) | u \in M \} \subset A \oplus A \oplus M$. For any $1$-cocycle $B : A \oplus A \rightarrow M$, by Proposition \ref{gauge-prop},
\begin{align*}
\tau_B (\mathrm{Gr}(R,S)) = \{ ( Ru, Su , u + B (Ru, Su)) | u\in M \} \subset A \oplus A \oplus M
\end{align*}
is a subalgebra. The subalgebra $\tau_B (\mathrm{Gr}(R,S))$ may not be the graph of a pair of some maps $(R, S')$. However, if the linear map $\mathrm{id} + B \circ (R, S) : M \rightarrow M$ is invertible, then $\tau_B (\mathrm{Gr}(R,S))$ is the graph of the pair of linear maps 
\begin{align*}
\big( R \circ (\mathrm{id} + B \circ (R, S))^{-1}, S \circ (\mathrm{id} + B \circ (R, S))^{-1}  \big) : M \rightarrow A \oplus A.
\end{align*}
In such a case, we call $B$ a $(R, S)$-admissible $1$-cocycle. Therefore, in this case, by Proposition \ref{rbs-graph}, the pair of linear maps $\big( R \circ (\mathrm{id} + B \circ (R, S))^{-1}, S \circ (\mathrm{id} + B \circ (R, S))^{-1}  \big)$ is a generalized Rota-Baxter system on $M$ over the algebra $A$. This generalized Rota-Baxter system is called the gauge transformation of $(R,S) $ associated with the $(R,S)$-admissible $1$-cocycle $B$.

\begin{remark}
If $R, S$ are both invertible, then $R_B, S_B$ are so. In fact, we have $R_B^{-1} (a) = (\mathrm{id} + B \circ (R, S)) (R^{-1} (a)) = R^{-1} (a) + B (a, SR^{-1} (a))$ and $S_B^{-1} (a) =  (\mathrm{id} + B \circ (R, S)) (S^{-1}(a)) = S^{-1} (a) + B (RS^{-1}(a), a)$.
\end{remark}

\begin{prop}
Let $(R,S)$ be a generalized Rota-Baxter system on $M$ over the algebra $A$ and $(R_B, S_B)$ be the gauge transformation associated with the $(R,S)$-admissible $1$-form $B$. Then associative structures on $M$ induced from generalized Rota-Baxter systems $(R, S)$ and $(R_B, S_B)$ are isomorphic.
\end{prop}

\begin{proof}
Consider the invertible linear map $(\mathrm{id} + B \circ (R,S)) : M \rightarrow M$. Then we have
\begin{align*}
&( \mathrm{id} + B \circ (R,S) ) (u ) *_B ( \mathrm{id} + B \circ (R,S) ) (v) \\
&= ( \mathrm{id} + B \circ (R,S) ) (u ) \cdot S(n) + R(u) \cdot ( \mathrm{id} + B \circ (R,S))(v) \\
&= u \cdot S(v) + R(u) \cdot  v + R(u) \cdot B (Rv, Sv) + B (Ru, Su) \cdot S(v) \\
&= u \cdot S(v) + R(u) \cdot v + (Ru, Su) \cdot B (Rv, Sv) + B ( Ru, Su) \cdot (Rv, Sv) \\
&= u \cdot S(v) + R(u) \cdot v + B ( (Ru,Su) (Rv,Sv)) ~~~\quad (\text{as } B \text{ is a } 1\text{-cocycle}) \\
&= u * v + (B \circ (R, S))( u * v ) = (\mathrm{id} + B \circ (R,S)) (u * v).
\end{align*}
Hence the proof.
\end{proof}

\subsection{Reductions}
Let $A$ be an associative algebra and $M$ an $A$-bimodule. Let $(R, S)$ be a generalized Rota-Baxter system on $M$ over the algebra $A$. Consider a subalgebra $B \subset A$ and a vector subspace $E \subset A$ with the property that the quotient $B / E \cap B$ is an associative algebra and the projection $\pi : B \rightarrow B / E \cap B$ is an algebra morphism.

Suppose $N \subset M$ is a $B$-bimodule. Define
\begin{align*}
(E \cap B)^0_N = \{ u \in N |~ a \cdot u = u \cdot a = 0, ~ \forall a \in E \cap B \}.
\end{align*}
Then it is easy to see that $(E \cap B)^0_N$ is an $B / E \cap B$-bimodule with $|b| \cdot u = b \cdot u$ and $u  \cdot |b| = u \cdot b$, for $|b| \in B / E \cap B$ and $u \in (E \cap B)^0_N$. Observe that $b \cdot u $ and $u \cdot b$ is in $(E \cap B)^0_N$. This follows as for any $a \in E \cap B$, we have
\begin{align*}
a \cdot (b \cdot u ) = (ab) \cdot u = 0 ~~ \text{ and } ~~ (b \cdot u) \cdot a = b \cdot (u \cdot a ) = 0.
\end{align*}
This shows that $b \cdot u \in (E \cap B)^0_N$.  Similarly, one can show that $u \cdot b \in (E \cap B)^0_N$.

\begin{defn}
Let $(R, S)$ be a generalized Rota-Baxter system on $M$ over the algebra $A$. A triple $(B, E, N)$ as above is said to be reducible if there is a generalized Rota-Baxter system $(\overline{R}, \overline{S}) : (E \cap B)^0_N \rightarrow B / E \cap B$ such that for any $u,v \in (E \cap B)_N^0$, we have 
\begin{align}\label{redu-cond}
\overline{R}(u) \cdot v = R(u) \cdot v, \quad v \cdot \overline{R}(u) = v \cdot R(u), \quad \overline{S}(u) \cdot v = S(u) \cdot v ~~~\text{ and }~~~ v \cdot \overline{S}(u) = v  \cdot S(u).
\end{align}
\end{defn}
The Marsden-Ratiu reduction theorem for generalized Rota-Baxter system is given by the following.

\begin{thm}
Let $(R, S)$ be a generalized Rota-Baxter system on $M$ over the algebra $A$. If $R ((E \cap B)^0_N ) \subset B $ and  $S ((E \cap B)^0_N ) \subset B $ then $(B, E, N)$ is reducible.
\end{thm}

\begin{proof}
We define $\overline{R}, \overline{S} : (E \cap B)_N^0 \rightarrow B / E \cap B$ by
\begin{align*}
\overline{R} (u) = |R(u)| ~~~ \text{ and } ~~~ \overline{S}(u) = | S(u)|.
\end{align*}
Then we have
\begin{align*}
\overline{R}(u) \overline{R}(v) = | R(u) || R(v)| = | R(u) R(v)|.
\end{align*}
On the other hand,
\begin{align*}
\overline{R} ( \overline{R} (u) \cdot v + u \cdot \overline{S}(v)) =~& \overline{R}  ( R(u) \cdot v + u \cdot S(v)) \\
=~& | R ( R(u) \cdot v + u \cdot S(v)) | = | R(u) R(v)|. 
\end{align*}
Hence $\overline{R}(u) \overline{R}(v) = \overline{R} ( \overline{R} (u) \cdot v + u \cdot \overline{S}(v))$. Similarly, we can show that $\overline{S}(u) \overline{S}(v) = \overline{S} ( \overline{R} (u)\cdot v + u \cdot \overline{S}(v))$. Hence $(\overline{R}, \overline{S})$ is a generalized Rota-Baxter system on $(E \cap B)^0_N$ over the algebra $B / E \cap B$. Moreover, the condition (\ref{redu-cond}) holds. Hence $(B, E, N)$ is reducible.
\end{proof}

As a consequence, we have the following.
\begin{corollary}
Let $(R, S)$ be a generalized Rota-Baxter system on $M$ over the algebra $A$. Let $B \subset A$ be a subalgebra and $N \subset M$ a $B$-bimodule. If $R (N) \subset B$ and $S(N) \subset B$, then the restrictions $(R, S) : N \rightarrow B$ is a generalized Rota-Baxter system on $N$ over $B$.
\end{corollary}

\begin{corollary} Let $(R, S)$ be a Rota-Baxter system on $M$ over $A$. For any ideal $E \subset A$, the space $E_M^0$ is an $A/ E$-bimodule and the pair of maps $(\overline{R}, \overline{S}) : E_M^0 \rightarrow A / E$ defined by  $\overline{R} (u) = |R(u)|$ and $\overline{S}(u) = | S(u)|$ is a generalized Rota-Baxter system on $E_M^0$ over $A/ E$.
\end{corollary}

\section{Maurer-Cartan characterization and cohomology}\label{sec-3}
In this section, we construct a graded Lie algebra whose Maurer-Cartan elements are precisely generalized Rota-Baxter systems. The graded Lie algebra is obtained from Voronov's derived bracket construction.

We first recall the Gerstenhaber bracket on multilinear maps on a vector space. Let $V$ be a vector space. For any $f \in \mathrm{Hom}(V^{\otimes m}, V)$ and $g \in \mathrm{Hom}(V^{\otimes n }, V)$, the Gerstenhaber bracket $[f, g ] \in \mathrm{Hom}(V^{\otimes m +n-1}, V)$ is given by $[f, g] = f \bullet g - (-1)^{(m-1)(n-1)} g \bullet f$, where 
\begin{align*}
(f \bullet g )(v_1, \ldots, v_{m+n-1}) = \sum_{i=1}^m (-1)^{(i-1)(n-1)} f( v_1, \ldots, v_{i-1}, g (v_i, \ldots, v_{i+n-1}), v_{i+n}, \ldots, v_{m+n-1}).
\end{align*}
Then the graded space $\oplus_{n}  \mathrm{Hom}(V^{\otimes n }, V)$ with the Gerstenhaber bracket $[~, ~]$ is a degree $-1$ graded Lie algebra.

Let $A$ and $M$ be two vector spaces. Suppose $\mu : A^{\otimes 2} \rightarrow A, ~ (a, b) \mapsto ab$; $l : A \otimes M \rightarrow M, ~ (a, u) \mapsto a \cdot u$ and $r : M \otimes A \rightarrow A,~ (u,a) \mapsto u \cdot a$ are linear maps. Consider the vector space $V = A \oplus A \oplus M$. Define an element
\begin{align*}
\mu + \mu + l_1 + r_2 \in \mathrm{Hom}(( A \oplus A \oplus M)^{\otimes 2}, A \oplus A \oplus M )   ~~~~~ \text{ by }
\end{align*}
\begin{align*}
(\mu + \mu + l_1 + r_2) (( a_1, a_2, u), (b_1, b_2, v)) = (a_1 b_1, a_2 b_2, a_1 \cdot v + u \cdot b_2).
\end{align*}

\begin{prop}
With the above notations, $\mu$ defines an associative product on $A$ and $(l, r)$ gives rise to an $A$-bimodule structure on $M$ if and only if $\mu + \mu + l_1 + r_2  \in \mathrm{Hom}(( A \oplus A \oplus M)^{\otimes 2}, A \oplus A \oplus M )$ is a Maurer-Cartan element with respect to the Gerstenhaber bracket.
\end{prop}

\begin{proof}
For any $(a_1, a_2, u), (b_1, b_2 , v)$ and $(c_1, c_2, w) \in A \oplus A \oplus M$, we have
\begin{align*}
&(\mu + \mu + l_1 + r_2) \big( (\mu + \mu + l_1 + r_2) ( (a_1, a_2, u), (b_1, b_2 , v )  ), (c_1, c_2, w) \big) \\
&= (\mu + \mu + l_1 + r_2) \big( (a_1 b_1, a_2 b_2 , a_1 \cdot  v + u \cdot b_2), (c_1, c_2, w) \big) \\
&= \big(   (a_1 b_1) c_1 , (a_2 b_2 ) c_2 , (a_1 b_1) \cdot w + (a_1 \cdot v) \cdot c_2 + (u \cdot b_2 ) \cdot c_2 \big). 
\end{align*}
On the other hand, 
\begin{align*}
&(\mu + \mu + l_1 + r_2) \big(   (a_1, a_2, u),  (\mu + \mu + l_1 + r_2) ( (b_1, b_2 , v ) , (c_1, c_2, w)  ) \big) \\
&= (\mu + \mu + l_1 + r_2) \big(    (a_1, a_2, u), ( b_1 c_1, b_2 c_2, b_1 \cdot w + v \cdot c_2) \big)\\
&= \big(  a_1 (b_1 c_1), a_2 (b_2 c_2), a_1 \cdot (b_1 \cdot w) + a_1 \cdot (v \cdot c_2 ) + u \cdot (b_2 c_2)   \big).
\end{align*}
Hence $ ( \mu + \mu + l_1 + r_2 ) \circ ( \mu + \mu + l_1 + r_2 ) = 0 $ if and only if $\mu$ is associative and $(l, r)$ defines an $A$-bimodule structure on $M$. 
\end{proof}

Thus it follows from the above Proposition that the graded vector space $\bigoplus \mathrm{Hom} (( A \oplus A \oplus M)^{\otimes 2}, A \oplus A \oplus M )$ with the differential $d_{\mu + \mu + l_1 + r_2} := [ \mu + \mu + l_1 + r_2, ~~ ]$ is a dgLa. Moreover, it follows from the definition of the bracket that $\bigoplus_{n \geq 1} \mathrm{Hom} (M^{\otimes n}, A \oplus A) $ is an abelian subalgebra. Therefore, by the derived bracket construction of Voronov \cite{voro} yields a graded Lie algebra bracket on $\bigoplus_{n \geq 1} \mathrm{Hom} (M^{\otimes n}, A \oplus A)$ given by
\begin{align}\label{deri-brk}
\llbracket ( P, Q) , (P', Q') \rrbracket := (-1)^m [[  \mu + \mu + l_1 + r_2, (P, Q) ], (P', Q') ],
\end{align}
for $( P, Q) \in \mathrm{Hom} (M^{\otimes m}, A \oplus A) ,~ (P', Q') \in \mathrm{Hom} (M^{\otimes n}, A \oplus A) $. Let $\mathrm{pr}_1, \mathrm{pr}_2 : A \oplus A \rightarrow A$ denote the projection maps onto the first and second factor, respectively. Explicitly, the bracket (\ref{deri-brk}) is given by
 \begin{align} \label{der-pr1}
&\mathrm{pr}_1 \big(  \llbracket ( P, Q) , (P', Q') \rrbracket (u_1, \ldots, u_{m+n}) \big) \\
&= \sum_{i=1}^m (-1)^{(i-1) n} P (u_1, \ldots, u_{i-1}, P' (u_i, \ldots, u_{i+n-1}) \cdot u_{i+n}, \ldots, u_{m+n} ) \nonumber\\
~&- \sum_{i=1}^m (-1)^{in} P (u_1, \ldots, u_{i-1}, u_i \cdot Q' (u_{i+1}, \ldots, u_{i+n}), u_{i+n+1}, \ldots, u_{m+n}) \nonumber\\
&- (-1)^{mn} \bigg\{ \sum_{i=1}^n (-1)^{(i-1)m}~ P' (u_1, \ldots, u_{i-1}, P (u_i, \ldots, u_{i+m-1}) \cdot u_{i+m}, \ldots, u_{m+n}) \nonumber\\
&- \sum_{i=1}^n (-1)^{im}~ P' (u_1, \ldots, u_{i-1}, u_i \cdot Q (u_{i+1}, \ldots, u_{i+m}) , u_{i+m+1}, \ldots, u_{m+n} ) \bigg\} \nonumber\\
&+ (-1)^{mn} \big[ P(u_1, \ldots, u_m) P' (u_{m+1}, \ldots, u_{m+n} ) - (-1)^{mn}~ P' (u_1, \ldots, u_n) P ( u_{n+1}, \ldots, u_{m+n} ) \big], \nonumber 
\end{align}
and
 \begin{align}\label{der-pr2}
&\mathrm{pr}_2 \big(  \llbracket ( P, Q) , (P', Q') \rrbracket  (u_1, \ldots, u_{m+n}) \big)\\
&= \sum_{i=1}^m (-1)^{(i-1) n} Q (u_1, \ldots, u_{i-1}, P' (u_i, \ldots, u_{i+n-1}) \cdot u_{i+n}, \ldots, u_{m+n} ) \nonumber\\
~&- \sum_{i=1}^m (-1)^{in} Q (u_1, \ldots, u_{i-1}, u_i \cdot Q' (u_{i+1}, \ldots, u_{i+n}), u_{i+n+1}, \ldots, u_{m+n}) \nonumber\\
&- (-1)^{mn} \bigg\{ \sum_{i=1}^n (-1)^{(i-1)m}~ Q' (u_1, \ldots, u_{i-1}, P (u_i, \ldots, u_{i+m-1}) \cdot u_{i+m}, \ldots, u_{m+n}) \nonumber\\
&- \sum_{i=1}^n (-1)^{im}~ Q' (u_1, \ldots, u_{i-1}, u_i \cdot Q (u_{i+1}, \ldots, u_{i+m}) , u_{i+m+1}, \ldots, u_{m+n} ) \bigg\} \nonumber\\
&+ (-1)^{mn} \big[ Q(u_1, \ldots, u_m) Q' (u_{m+1}, \ldots, u_{m+n} ) - (-1)^{mn}~ Q' (u_1, \ldots, u_n) Q ( u_{n+1}, \ldots, u_{m+n} ) \big], \nonumber
\end{align}
for $( P, Q) \in  \mathrm{Hom} (M^{\otimes m}, A \oplus A), (P', Q') \in \mathrm{Hom} (M^{\otimes n}, A \oplus A), ~ m, n \geq 1.$ We can extend this bracket to the graded space $\bigoplus_{n \geq 0} \mathrm{Hom}(M^{\otimes n}, A \oplus A)$ by the following
\begin{align}
\mathrm{pr}_1 \big( \llbracket ( P, Q) , (a, b) \rrbracket (u_1, \ldots, u_{m}) \big) =  ~& \sum_{i=1}^m P (u_1, \ldots, u_{i-1}, a \cdot u_i - u_i \cdot  b , u_{i+1}, \ldots, u_{m} ) \\
~~&+ P (u_1, \ldots, u_m)  a - a P (u_1, \ldots, u_m), \nonumber
\end{align}
\begin{align}
\mathrm{pr}_2 \big( \llbracket ( P, Q) , (a, b) \rrbracket (u_1, \ldots, u_{m}) \big) =  ~& \sum_{i=1}^m Q (u_1, \ldots, u_{i-1}, a \cdot u_i - u_i \cdot b , u_{i+1}, \ldots, u_{m} ) \\
~~&+ Q (u_1, \ldots, u_m) b - b Q (u_1, \ldots, u_m), \nonumber
\end{align}
and finally
\begin{align}\label{der-pr5}
\llbracket (a,b), (c,d) \rrbracket = (ac-ca, bd-db).
\end{align}

\begin{thm}\label{thm-rbs-mc}
Let $A$ be an associative algebra and $M$ an $A$-bimodule. Then the graded vector space $C^* (M, A) = \bigoplus_{n \geq 0}   \mathrm{Hom} (M^{\otimes n}, A \oplus A)  $ with the bracket $\llbracket ~~, ~~ \rrbracket$ forms a graded Lie algebra. A pair $(R, S)$ of linear maps from $M$ to $A$ is a generalized Rota-Baxter system if and only if $(R, S) \in C^1 (M, A)$ is a Maurer-Cartan element in the graded Lie algebra $(C^* (M, A), \llbracket ~~, ~~ \rrbracket )$.
\end{thm}

\begin{proof}
The first part follows from previous discussions. For the second part, it follows from (\ref{der-pr1}) and (\ref{der-pr2}) that if $(R,S) \in \mathrm{Hom}(M, A \oplus A)$ then 
\begin{align*}
\mathrm{pr}_1 ( \llbracket (R,S), (R,S) \rrbracket (u,v) ) =~&  2 \big( R (R(u) \cdot  v + u \cdot S(v)) - R(u) R(v)    \big), \\
\mathrm{pr}_2 ( \llbracket (R,S), (R,S) \rrbracket (u,v) ) =~&  2 \big( S (R(u) \cdot  v + u \cdot S(v)) - S(u) S(v)    \big).
\end{align*}
Hence $(R,S)$ is a Maurer-Cartan element (i.e. $\llbracket (R,S), (R,S) \rrbracket =0$) if and only if $(R,S)$ is a generalized Rota-Baxter system.
\end{proof}

Thus, generalized Rota-Baxter systems can be characterized as Maurer-Cartan elements in a gLa. It follows from the above theorem that if $(R, S)$ is a generalized Rota-Baxter system, then $d_{(R, S)} := \llbracket (R, S), ~~ \rrbracket $
is a differential on $C^* (M, A)$ and makes the gLa $(C^\bullet (M, A), \llbracket ~~, ~~ \rrbracket )$ into a dgLa.

The cohomology of the cochain complex $(C^\bullet (M,A), d_{(R,S)})$ is called the cohomology of the generalized Rota-Baxter system $(R,S)$. We denote the corresponding cohomology groups simply by $H^\bullet (M,A)$. In the next subsection, we view this cohomology as the Hochschild cohomology of $(M, *)$ with coefficients in a suitable representation on $A \oplus A$.

The proof the following result is standard in the study of Maurer-Cartan elements in a gLa.

\begin{thm}\label{sum-rbs}
Let $(R, S)$ be a generalized Rota-Baxter system on $M$ over the algebra $A$. Then for any pair $(R', S')$ of linear maps from $M$ to $A$, the sum $(R+ R', S + S')$ is a generalized Rota-Baxter system if and only if $(R', S')$ is a Maurer-Cartan element in the dgLa $(C^\bullet (M, A), \llbracket ~~, ~~ \rrbracket, d_{(R, S)} ),$ i.e.
\begin{align*}
\llbracket (R+ R', S + S'), (R+ R', S + S') \rrbracket = 0  ~\Leftrightarrow ~ d_{(R, S)} (R', S') + \frac{1}{2} \llbracket (R', S'), (R', S')  \rrbracket = 0.
\end{align*}
\end{thm}

\subsection{Hochschild cohomology}
Let $(R,S)$ be a generalized Rota-Baxter system on $M$ over the algebra $A$. Consider the associative algebra $(M, *)$ where $u * v = R(u) \cdot v + u \cdot S(v)$, for $u, v \in M$.

The following result generalizes \cite[Lemma 2.11]{uchino}.

\begin{prop}\label{prop-hoch}
The associative algebra $(M, *)$ has a bimodule representation on $A \oplus A$ given by
\begin{align*}
l (u, (a,b)) =~& (R(u) a - R(u \cdot b), S(u) b - S(u \cdot b) ),\\
r ((a,b), u) =~& (a R(u) - R(a \cdot u), b S(u) - S(a \cdot u) ),
\end{align*}
for $u \in M$ and $(b,c) \in A \oplus A$.
\end{prop}

\begin{proof}
The result follows from a straightforward calculation.
\end{proof}

It follows from the above lemma that we can consider the Hochschild cohomology of the associative algebra $(M, *)$ with coefficients in the bimodule $A \oplus A$. More precisely, the Hochschild complex is given by $(C^\bullet_{\mathrm{Hoch}} (M, A \oplus A), \delta_{\mathrm{Hoch}} )$ where
$C^n_{\mathrm{Hoch}} (M, A \oplus A) = \mathrm{Hom}(M^{\otimes n}, A \oplus A), \text{ for } n \geq 0$
and the differential $\delta_{\mathrm{Hoch}} : C^n_{\mathrm{Hoch}} (M, A \oplus A) \rightarrow C^{n+1}_{\mathrm{Hoch}} (M, A \oplus A)$ given by
\begin{align*}
\mathrm{pr}_1 (\delta_{\mathrm{Hoch}} (f,g)) =~& R ( u_1) f (u_2, \ldots, u_{n+1})  - R (u_1  \cdot g(u_2, \ldots, u_{n+1})) \\
~&+ \sum_{i=1}^n (-1)^i ~ f (u_1, \ldots, u_{i-1}, R(u_i) \cdot u_{i+1} + u_i \cdot S( u_{i+1}) , \ldots, u_{n+1}) \\
~&+ (-1)^{n+1}~  f (u_1, \ldots, u_n ) R(u_{n+1})  - (-1)^{n+1} R ( f(u_1, \ldots, u_n) \cdot u_{n+1} ).  \\
\mathrm{pr}_2 (\delta_{\mathrm{Hoch}} (f,g)) =~& S ( u_1) g (u_2, \ldots, u_{n+1})  - S (u_1 \cdot g(u_2, \ldots, u_{n+1})) \\
~&+ \sum_{i=1}^n (-1)^i ~ g (u_1, \ldots, u_{i-1}, R(u_i) \cdot u_{i+1} + u_i \cdot S( u_{i+1}) , \ldots, u_{n+1}) \\
~&+ (-1)^{n+1}~  g (u_1, \ldots, u_n ) S(u_{n+1})  - (-1)^{n+1} S ( f(u_1, \ldots, u_n) \cdot u_{n+1} ).
\end{align*}

\begin{prop}
Let $(R,S)$ be a generalized Rota-Baxter system on $M$ over the algebra $A$. Then
\begin{align*}
d_{(R,S)} (f,g) = (-1)^n \delta_{\mathrm{Hoch}} (f,g),~ \text{ for } (f,g) \in C^n_{\mathrm{Hoch}}(M, A \oplus A).
\end{align*}
\end{prop}
\begin{proof}
For $(f,g) \in C^n_{\mathrm{Hoch}}(M, A \oplus A)$ and $u_1, \ldots, u_{n+1} \in M$, we have from (\ref{der-pr1}) that
\begin{align*}
&\mathrm{pr}_1 \big( \llbracket (R,S),(f,g) \rrbracket (u_1, \ldots, u_{n+1}) \big) \\
&= R ( f(u_1, \ldots, u_n) \cdot u_{n+1}) - (-1)^n R ( u_1 \cdot g(u_2, \ldots, u_{n+1})) \\
& - (-1)^n \sum_{i=1}^n (-1)^i ~ f (u_1, \ldots, u_{i-1}, R(u_i) \cdot u_{i+1}+ u_i \cdot S(u_{i+1}), \ldots, u_{n+1}) \\
& + (-1)^n R(u_1) f (u_2, \ldots, u_{n+1}) - f (u_1, \ldots, u_n) R(u_{n+1}) \\
&= (-1)^n~ \mathrm{pr}_1 \big( \delta_{\mathrm{Hoch}} (f,g) (u_1, \ldots, u_{n+1}) \big). 
\end{align*}
Similarly, we can show that $\mathrm{pr}_2 \big( \llbracket (R,S),(f,g) \rrbracket (u_1, \ldots, u_{n+1}) \big) = (-1)^n~ \mathrm{pr}_2 \big( \delta_{\mathrm{Hoch}} (f,g) (u_1, \ldots, u_{n+1}) \big)$. Hence, we get $d_{(R,S)} (f,g) = \llbracket (R,S) , (f,g) \rrbracket = (-1)^n \delta_{\mathrm{Hoch}} (f,g)$.
\end{proof}

It follows from the above proposition that the cohomology of a generalized Rota-Baxter system induced from the Maurer-Cartan element coincides with the above Hochschild cohomology.

\subsection{Relation with dendriform algebra cohomology}\label{subsec-dend-rel}
Cohomology of dendriform algebras was first introduced by Loday \cite{loday} for trivial coefficients. Later on, an operadic approach of the cohomology was given in the book by Loday and Vallette \cite{lod-val-book}. The explicit description of the cohomology can be found in \cite{das1}. In this subsection, we find the relationship between the cohomology of a generalized Rota-Baxter system and the cohomology of the corresponding dendriform algebra.

For each $n \geq 1$, let $C_n$ be the set of first $n$ natural numbers. We denote the elements of $C_n$ by $\{ [1], [2], \ldots, [n]\}$ for convenience. Then it has been shown in \cite{das1} that for any vector space $D$, the collection of 
spaces
\begin{align*}
\mathcal{O}(n) := \mathrm{Hom}(\mathbb{K}[C_n] \otimes D^{ \otimes n } , D),~\text{ for } n \geq 1
\end{align*}
forms a non-symmetric operad with partial compositions

$(f \circ_i g)([r]; a_1, \ldots, a_{m+n-1}) =$
\begin{align}\label{dend-partial-comp}
\begin{cases}
f ([r]; a_1, \ldots, a_{i-1}, g ([1]+ \cdots + [n]; a_i, \ldots, a_{i+n-1}), \ldots, a_{m+n-1}) ~& \text{ if } 1 \leq r \leq i-1 \\
f ([i]; a_1, \ldots, a_{i-1}, g ([r-i+1]; a_i, \ldots, a_{i+n-1}), \ldots, a_{m+n-1}) ~& \text{ if } i \leq r \leq i+ n-1\\
f ([r-n+1]; a_1, \ldots, a_{i-1}, g ([1]+ \cdots + [n]; a_i, \ldots, a_{i+n-1}), \ldots, a_{m+n-1}) ~& \text{ if } i+ n \leq r \leq m + n -1,
\end{cases}
\end{align}
for $f \in \mathcal{O}(m), ~ g \in \mathcal{O}(n),~ 1 \leq i \leq m$ and $[r] \in C_{m+n-1}$. Therefore, by a result of Gerstenhaber and Voronov \cite{gers-voro}, the graded space $\mathcal{O} (\bullet + 1) = \oplus_{n \geq 1} \mathcal{O}( n +1 )$ carries a graded Lie bracket
\begin{align}\label{llceil-br}
\llceil f, g \rrceil = \sum_{i=1}^{m+1} (-1)^{(i-1)n} f \circ_i g ~-~ (-1)^{mn} \sum_{i=1}^{n+1} (-1)^{(i-1)m} g \circ_i f,
\end{align}
for $f \in \mathcal{O}(m+1)$ and $g \in \mathcal{O}(n+1)$. If $(D, \prec, \succ)$ is a dendriform algebra, then the element $\pi \in \mathcal{O}(2)$ defined by 
\begin{align*}
\pi ([1];a,b) = a \prec b ~~~~ \text{ and } ~~~~ \pi ([2];a,b) = a \succ b 
\end{align*}
is a Maurer-Cartan element in the graded Lie algebra $(\mathcal{O}(\bullet +1) , \llceil ~, ~ \rrceil )$. Therefore, $\pi$ defines a differential $\delta_\pi (f ) = (-1)^{n-1} \llceil \pi, f \rrceil$, for $f \in \mathcal{O}(n)$. The corresponding cohomology groups are called the cohomology of the dendriform algebra $(D, \prec, \succ)$ with coefficients in itself. We denote the cohomology groups by $H^\bullet_{\mathrm{dend}} (D,D).$

\medskip

Let $(R,S)$ be a generalized Rota-Baxter system on an $A$-bimodule $M$. Consider the corresponding dendriform algebra $(M, \prec, \succ)$ given in Proposition \ref{RBS-dend}. Let $\pi \in \mathrm{Hom}(\mathbb{K}[C_2] \otimes M^{\otimes 2}, M)$ denotes the corresponding Maurer-Cartan element, i.e.
\begin{align*}
\pi_M ( [1]; u,v ) = u \cdot S(v) ~~~~ \text{ and } ~~~~ \pi_M ([2]; u,v) = R(u) \cdot v, ~ \text{ for } u,v \in M.
\end{align*}
We define a collection of maps
\begin{align*}
\Theta_n : \mathrm{Hom} (M^{\otimes n}, A \oplus A ) \rightarrow \mathrm{Hom}( \mathbb{K}[C_{n+1}] \otimes M^{\otimes n+1}, M) ~~~\text{ by }
\end{align*}
\begin{align}\label{theta-def}
(\Theta_n (P, Q )) ([r]; u_1, u_2, \ldots, u_{n+1}) = \begin{cases} (-1)^{n+1} ~u_1 \cdot Q(u_2, \ldots u_{n+1})  ~~~& \text{if }~ r = 1 \\
0 ~~~& \text{if } ~2 \leq r \leq n \\
P(u_1, \ldots, u_n ) \cdot u_{n+1} ~~~& \text{if }~ r = n + 1. \end{cases}
\end{align}

We have the following lemma which is crucial for the next theorem.
\begin{lemma}
The collection $\{ \Theta_n \}$ of maps preserve the respective graded Lie brackets, i.e.
\begin{align*}
\llceil \Theta_m (P,Q), \Theta_n (P',Q') \rrceil = \Theta_{m+n} \llbracket (P,Q), (P',Q') \rrbracket,
\end{align*} 
for $(P,Q) \in \mathrm{Hom}(M^{\otimes m}, A \oplus A)$ and $(P', Q') \in \mathrm{Hom}(M^{\otimes n}, A \oplus A).$
\end{lemma}

\begin{proof}
For $u_0, u_1, \ldots, u_{m+n} \in M$, we have from (\ref{llceil-br}) that
\begin{align*}
&\llceil \Theta_m (P, Q), \Theta_n (P',Q') \rrceil ([1]; u_0, u_1, \ldots, u_{m+n}) \\
&= \bigg( \sum_{i=1}^{m+1} (-1)^{(i-1)n} \Theta_m (P,Q) \circ_i \Theta_n (P',Q') ~-~ (-1)^{mn} \sum_{i=1}^{n+1} (-1)^{(i-1)m} \Theta_n (P',Q') \circ_i \Theta_m (P,Q) \bigg) ([1]; u_0, \ldots, u_{m+n}) \\
&= \Theta_m (P,Q) ( [1]; \Theta_n (P',Q') ([1]; u_0, \ldots, u_n), u_{n+1}, \ldots, u_{m+n} )\\
&+ \sum_{i=1}^m (-1)^{in} \Theta_m (P,Q) ([1]; u_0, \ldots, u_{i-1}, \Theta_n (P',Q') ([1] + \cdots + [n+1]; u_i, \ldots, u_{i+n}), u_{i+n+1}, \ldots, u_{m+n}) \\
&- (-1)^{mn} \big\{   \Theta_n (P',Q') ( [1]; \Theta_m (P,Q) ([1]; u_0, \ldots, u_m), u_{m+1}, \ldots, u_{m+n} )\\
&+ \sum_{i=1}^n (-1)^{im} \Theta_n (P',Q') ([1]; u_0, \ldots, u_{i-1}, \Theta_m (P,Q) ([1] + \cdots + [m+1]; u_i, \ldots, u_{i+m}), u_{i+m+1}, \ldots, u_{m+n})  \big\}\\
&= (\Theta_{m+n} \llbracket (P, Q),(P',Q') \rrbracket)([1]; u_0, u_1, \ldots, u_{m+n}).
\end{align*}
Similarly, for $u_1, \ldots, u_{m+n+1} \in M$,
\begin{align*}
&\llceil \Theta_m (P,Q), \Theta_n (P',Q') \rrceil ([m+n+1]; u_1, \ldots, u_{m+n+1}) \\
&= \bigg( \sum_{i=1}^{m+1} (-1)^{(i-1)n} \Theta_m (P,Q) \circ_i \Theta_n (P',Q') ~-~ (-1)^{mn} \sum_{i=1}^{n+1} (-1)^{(i-1)m} \Theta_n (P',Q') \circ_i \Theta_m (P,Q) \bigg) \\& \hspace*{10cm} \hfill{([m+n+1]; u_1, \ldots, u_{m+n+1})}\\
&= \sum_{i=1}^m (-1)^{(i-1)n} \Theta_m (P,Q) ( [m+1] ; u_1, \ldots, u_{i-1}, \Theta_n (P',Q') ([1] + \cdots + [n+1]; u_i, \ldots, u_{i+n}), \ldots, u_{m+n+1})\\
&+ (-1)^{mn}  \Theta_m (P,Q) ([m+1]; u_1, \ldots, u_m, \Theta_n (P',Q') ([n+1]; u_{m+1}, \ldots, u_{m+n+1})) \\
&- (-1)^{mn} \big\{    \sum_{i=1}^n (-1)^{(i-1)m} \Theta_n (P',Q') ( [n+1] ; u_1, \ldots, u_{i-1}, \Theta_m (P,Q) ([1] + \cdots + [m+1]; u_i, \ldots, u_{i+m}), \ldots, u_{m+n+1})\\
&+ (-1)^{mn}  \Theta_n (P',Q') ([n+1]; u_1, \ldots, u_n, \Theta_m (P,Q) ([m+1]; u_{n+1}, \ldots, u_{m+n+1})) \big\} \\
&= (\Theta_{m+n} \llbracket P, Q \rrbracket ) ([m+n+1]; u_1, \ldots, u_{m+n+1}).
\end{align*}
For $2 \leq r \leq m+n$, it follows from (\ref{llceil-br}) and partial compositions (\ref{dend-partial-comp}) that 
\begin{align*}
\llceil \Theta_m (P,Q), \Theta_n (P',Q') \rrceil ([r]; u_0, \ldots, u_{m+n}) = 0 =  (\Theta_{m+n} \llbracket (P, Q),(P',Q') \rrbracket ) ([r]; u_0, \ldots, u_{m+n}).
\end{align*}
Hence the result follows.
\end{proof}

\begin{thm}
Let $(R,S)$ be a generalized Rota-Baxter system on $M$ over the algebra $A$. Then the collection of maps $\{ \Theta_n \}$ induces a morphism $\Theta_\bullet : H^\bullet (M,A) \rightarrow H^{\bullet +1}_{\mathrm{dend}} (M,M)$ from the cohomology of the generalized Rota-Baxter system $(R,S)$ to the cohomology of the corresponding dendriform algebra $M$.
\end{thm}

\begin{proof}
It follows from  (\ref{theta-def}) that $\Theta_1 (R,S) = \pi_M$. Therefore, by the above lemma, the following diagram commutes
\[
\xymatrix{
\mathrm{Hom}(M^{\otimes n }, A) \oplus \mathrm{Hom}(M^{\otimes n }, A) \ar[r]^{ (-1)^n \llbracket (R,S), ~ \rrbracket } \ar[d]_{\Theta_n} &  \mathrm{Hom}(M^{\otimes n+1 }, A) \oplus \mathrm{Hom}(M^{\otimes n+1 }, A)  \ar[d]^{\Theta_{n+1}}\\
\mathrm{Hom}(\mathbb{K}[C_{n+1}] \otimes M^{\otimes n+1 }, M) \ar[r]_{\delta_{\pi_M}} &  \mathrm{Hom}(\mathbb{K}[C_{n+2}] \otimes M^{\otimes n+2 }, M).
}
\]
Hence the result follows.
\end{proof}

\section{Deformations of generalized Rota-Baxter systems}\label{sec-4}
Formal deformation theory of associative Rota-Baxter operators was studied by the author in \cite{das2} along the line of Gerstenhaber \cite{gers}. In this section, we study formal one-parameter deformations of generalized Rota-Baxter systems.

Let $A$ be an associative algebra. Consider the space $A [[t]]$ of formal power series in $t$ with coefficients from $A$. The associative multiplication of $A$ induces an associative product on $A[[t]]$ by $\mathbb{K}[[t]]$-bilinearity. Moreover, if $M$ is an $A$-bimodule then $M[[t]]$ is an $A[[t]]$-bimodule by $\mathbb{K}[[t]]$-bilinearity.

\begin{defn}
Let $(R, S)$ be a generalized Rota-Baxter system on $M$ over the algebra $A$. A formal one-parameter deformation of $(R, S)$ consists of formal sums
\begin{align*}
R_t = R_0 + t R_1 + t^2 R_2 + \cdots ~\in \mathrm{Hom}( M, A) [[t]] \qquad S_t = S_0 + t S_1 + t^2 S_2 + \cdots ~\in \mathrm{Hom}( M, A) [[t]]
\end{align*}
with $R_0 = R, ~ S_0 = S$ and such that $(R_t, S_t)$ is a generalized Rota-Baxter system on $M[[t]]$ over the algebra $A[[t]]$. 
\end{defn}

Thus the followings are hold
\begin{align*}
R_t (u) R_t (v) = R_t  ( R_t(u) \cdot v + u \cdot S_t (v) )  ~~~ \text{ and } ~~~
S_t (u) S_t (v) = S_t  ( R_t(u) \cdot v + u \cdot S_t (v)).
\end{align*}
By expanding these equations and comparing coefficients of various powers of $t$, we obtain for $n \geq 0$,
\begin{align}
\sum_{i+j = n} R_i (u) R_j (v) =~& \sum_{i+j = n} R_i ( R_j (u) \cdot v + u \cdot S_j (v)), \label{def-eqn-1} \\
\sum_{i+j = n} S_i (u) S_j (v) =~& \sum_{i+j = n} S_i ( R_j (u) \cdot v + u \cdot S_j (v)). \label{def-eqn-2}
\end{align}
Both of these identities hold for $n = 0$ as $(R, S)$ is a generalized Rota-Baxter system. However, for $n=1$, we obtain
\begin{align*}
R(u) R_1 (v) + R_1 (u) R(v) =~& R ( R_1 (u) \cdot v + u \cdot S_1(v)) + R_1 ( R(u) \cdot v + u  \cdot S(v)),\\
S(u) S_1 (v) + S_1 (u) S(v) =~& S ( R_1 (u) \cdot v + u \cdot S_1(v)) + S_1 ( R(u) \cdot v + u \cdot S(v)).
\end{align*}
These identities are equivalent to
\begin{align*}
\llbracket (R, S), (R_1, S_1) \rrbracket = 0.
\end{align*}
In other words, the pair $(R_1, S_1)$ is a $1$-cocycle in the cohomology of the generalized Rota-Baxter system $(R, S)$. It is called the infinitesimal of the deformation. In general, if $(R_1, S_1) = \cdots = (R_{k-1}, S_{k-1}) = 0$ then $(R_k, S_k)$ is a $1$-cocycle in the cohomology of $(R,S)$.

\begin{defn} (Equivalence) Let $(R_t, S_t)$ and $(R_t', S_t')$ be two formal one-parameter deformations of a generalized Rota-Baxter system $(R,S)$. They are said to be equivalent if there exists $a, b \in A$; linear maps $\phi_i, \varphi_i \in \mathrm{End}(A)$ and $\psi_i \in \mathrm{End}(M)$ for $i \geq 2$ such that
\begin{align*}
\big( \phi_t = \mathrm{id}_A + t (l^\mathrm{ad}_a - r^\mathrm{ad}_a) + \sum_{i \geq 2} t^i \phi_i,~
  \varphi_t = \mathrm{id}_A + t (l^\mathrm{ad}_b - r^\mathrm{ad}_b) + \sum_{i \geq 2} t^i \varphi_i,~
  \psi_t = \mathrm{id}_M + t (l_a - r_b) + \sum_{i \geq 2} t^i \psi_i \big)
\end{align*}
is a morphism of generalized Rota-Baxter systems from $(R_t, S_t)$ to $(R_t', S_t')$.
\end{defn}

Thus it follows that the conditions of Definition \ref{grbs-mor} hold. In particular, we have $R_t' \circ \psi_t = \phi_t \circ R_t$ and $S_t' \circ \psi_t = \varphi_t \circ S_t$. Comparing coefficients of $t$ in both identities, we get
\begin{align*}
R_1 (u) - R_1' (u) =~& R(u) a - R(u \cdot b) - a  R(u) + R (a \cdot u) = \mathrm{pr}_1 (\delta_{\mathrm{Hoch}} (a,b) (u)),\\
S_1 (u) - S_1' (u) =~& S(u) b - S(u \cdot b) - b  S(u) + S (a \cdot u) = \mathrm{pr}_2 (\delta_{\mathrm{Hoch}} (a,b) (u)).
\end{align*}
Thus, we have $(R_1, S_1)(u) - (R_1', S_1')(u) = (\delta_{\mathrm{Hoch}} (a,b)) (u)$.
In summary, we obtain the following.

\begin{thm}
The infinitesimal of a deformation of a generalized Rota-Baxter system is a $1$-cocycle and the corresponding cohomology class depends only on the equivalence class of the deformation.
\end{thm}



The following result relates the deformation of a generalized Rota-Baxter system and the deformation of the dendriform structure in the sense of \cite{das1}.

\begin{prop}
Let $(R_t, S_t)$ be a formal one-parameter deformation of a generalized Rota-Baxter system $(R,S)$. Then $(\prec_t, \succ_t)$ defines a deformation of the induced dendriform algebra $(M, \prec, \succ)$ given in Proposition \ref{RBS-dend}, where $u \prec_t v = \sum_{i=0}^\infty t^i (u S_i(v))$ and $u \succ_t v = \sum_{i=0}^\infty t^i (R_i(u) v)$, for $u,v \in M.$
\end{prop}

\subsection{Extensions of finite order deformations}
Let $(R, S)$ be a generalized Rota-Baxter system on $M$ over the algebra $A$. Consider the space $A [[t]] / (t^{N+1})$ that inherits an associative algebra structure over $\mathbb{K}[[t]] /(t^{N+1})$ and $M[[t]] / (t^{N+1})$ is a bimodule over it.

\begin{defn}
A deformation of order $N$ of the generalized Rota-Baxter system  $(R, S)$ consist of finite sums $R_t = \sum_{i=0}^N t^i R_i $ (with $R_0 = R$) and $S_t =  \sum_{i=0}^N t^i S_i$ (with $S_0 = S$) such that $(R_t, S_t)$ is a generalized Rota-Baxter system on $M[[t]] / (t^{N+1})$ over the algebra $A [[t]] / (t^{N+1})$.
\end{defn}

Thus in a deformation of order $N$, the identities (\ref{def-eqn-1}) and (\ref{def-eqn-2}) hold for $n=0, 1, \ldots, N$. These identities can be equivalently expressed as 
\begin{align}\label{N-def}
\llbracket (R, S), (R_n, S_n ) \rrbracket = - \frac{1}{2} \sum_{i+j = n, i, j \geq 1} \llbracket (R_i, S_i), (R_j, S_j) \rrbracket, ~ \text{ for } n = 0, 1, \ldots, N.
\end{align}
\begin{defn}
A deformation $(R_t, S_t)$ of order $N$ is said to be extensible if there exists a pair $(R_{N+1}, S_{N+1})$ of linear maps from $M$ to $A$ such that $(R_t' = R_t + t^{N+1} R_{N+1}, S_t' = S_t + t^{N+1} S_{N+1} )$ is a deformation of order $N+1$.
\end{defn}

In such a case, one more deformation needs to be satisfied, namely,
\begin{align*}
\llbracket (R, S), (R_{N+1}, S_{N+1} ) \rrbracket = - \frac{1}{2} \sum_{i+j = N+1, i, j \geq 1} \llbracket (R_i, S_i), (R_j, S_j) \rrbracket.
\end{align*}
Note that on the right hand side of the above equation does not involve $R_{N+1}$ or $S_{N+1}$. Hence it depends only on the order $N$ deformation $(R_t, S_t)$. It is called the `obstruction' to extend the deformation $(R_t, S_t)$ and is denoted by $Ob_{(R_t, S_t)}$.

\begin{prop}
The obstruction $Ob_{(R_t, S_t)}$ is a $2$-cocycle in the cohomology of the generalized Rota-Baxter system $(R, S)$.
\end{prop}

\begin{proof}
We have
\begin{align*}
d_{(R, S)} (  Ob_{(R_t, S_t)} ) =~& - \frac{1}{2} \sum_{i+j = N+1, i, j \geq 1} \llbracket (R, S), \llbracket (R_i, S_i), (R_j, S_j) \rrbracket \rrbracket \\
=~& - \frac{1}{2} \sum_{i+j = N+1, i, j \geq 1}   \big( \llbracket \llbracket (R, S), (R_i, S_i) \rrbracket, (R_j, S_j) \rrbracket - \llbracket (R_i, S_i ), \llbracket (R, S), (R_j, S_j) \rrbracket \rrbracket  \big)  \\
=~& \frac{1}{4} \sum_{i_1+i_2+j = N+1, i_1, i_2, j \geq 1} \llbracket \llbracket (R_{i_1}, S_{i_1}), (R_{i_2}, S_{i_2}) \rrbracket, (R_j, S_j) \rrbracket \\
&- \frac{1}{4} \sum_{i+j_1 +j_2 = N+1, i, j_1, j_2 \geq 1} \llbracket (R_i, S_i ), \llbracket (R_{j_1}, S_{j_1}), (R_{j_2}, S_{j_2}) \rrbracket \rrbracket    \quad (\text{by } (\ref{N-def})) \\
=~& \frac{1}{2} \sum_{i+j+k = N+1, i, j, k \geq 1} \llbracket \llbracket (R_{i}, S_{i}), (R_{j}, S_{j}) \rrbracket, (R_k, S_k) \rrbracket = 0.
\end{align*}
Hence the proof.
\end{proof}

This shows that the obstruction gives rise to a second cohomology class $[Ob_{(R_t, S_t)}] \in H^2 (M, A) $, called the obstruction class.

As a consequence of (\ref{N-def}) we obtain the following.
\begin{thm}
A deformation $(R_t, S_t)$ of order $N$ is extensible if and only if the obstruction class $[Ob_{(R_t, S_t)}] \in H^2 (M, A) $ is trivial.
\end{thm}

\begin{thm}
If $H^2(M, A) = 0$ then every finite order deformation of $(R, S)$ is extensible.
\end{thm}

\medskip

\section{Rota-Baxter systems, Yang-Baxter pairs, covariant bialgebras and averaging systems}\label{sec-5}
In this section, we give some further study of Rota-Baxter systems, associative Yang-Baxter pairs and covariant bialgebras introduced by Brzezi\'{n}ski \cite{brz}. Some applications of the results of previous sections are given to these structures. Finally, we introduce generalized averaging systems and observe that they are particular cases of generalized Rota-Baxter systems. A generalized averaging system induces an associative dialgebra structure in the sense of Loday \cite{loday}.

\subsection{Rota-Baxter systems}
It has been shown in Section \ref{sec-2} that a Rota-Baxter system on an associative algebra $A$ is a generalized Rota-Baxter system on the adjoint bimodule over the algebra $A$. Thus, the results of the previous sections can be applied to Rota-Baxter systems by considering the adjoint bimodule.

Combining Theorem \ref{thm-rbs-mc} and Theorem \ref{sum-rbs} for the adjoint bimodule, we get the following.

\begin{thm}
Let $A$ be an associative algebra. Then there is a graded Lie bracket $\llbracket ~, ~ \rrbracket$ defined by (\ref{der-pr1})-(\ref{der-pr5}) on the graded vector space $C^\bullet (A, A) = \bigoplus_{n \geq 0}  \mathrm{Hom}(A^{\otimes n}, A \oplus A)$. A pair $(R, S)$ of linear maps on $A$ is a Rota-Baxter system if and only if $(R, S) \in C^1(A, A)$ is a Maurer-Cartan element in the graded Lie algebra $(C^\bullet (A, A), \llbracket ~, ~ \rrbracket )$. Consequently, a Rota-Baxter system $(R,S)$ induces a differential $d_{(R,S)}  = \llbracket (R,S), ~ \rrbracket$ on the graded vector space $C^\bullet (A, A)$ which makes $(C^\bullet (A, A), \llbracket ~, \rrbracket, d_{(R,S)} )$ a dgLa.

Moreover, for any pair $(R', S')$ of linear maps on $A$, the pair $(R+R', S+S')$ is a Rota-Baxter system on $A$ if and only if $(R', S')$ satisfies
\begin{align*}
d_{(R,S)} (R', S') + \frac{1}{2} \llbracket (R',S'), (R', S') \rrbracket = 0.
\end{align*}
\end{thm}

Moreover, Proposition \ref{prop-hoch} leads to the following when we consider the adjoint bimodule.

\begin{prop}
Let $A$ be an associative algebra and $(R,S)$ be a Rota-Baxter system on $A$. Then the vector space $A$ carries a new associative algebra structure with product $a * b = R(a) b + a S(b)$, for $a, b \in A$. The associative algebra $(A, *)$ has a bimodule representation on $A \oplus A$ given by
\begin{align*}
l (a, (b,c)) = (R(a)b - R(ac) , S(a) c - S(ac) ) ~~~ \text{ and } ~~~ r( (b,c), a) = (b R(a) - R (ba), c S(a) - S(ba)).
\end{align*}
\end{prop}

The Hochschild cohomology of the associative algebra $(A, *)$ with coefficients in the above bimodule $A \oplus A$ is isomorphic to the cohomology induced from the Maurer-Cartan element $(R,S)$ in the graded Lie algebra $(C^\bullet(A,A), \llbracket ~,~ \rrbracket)$. This cohomology is called the cohomology of the Rota-Baxter system $(R, S)$ on the algebra $A$.

Like the previous section, we may define deformations of a Rota-Baxter system $(R,S)$ on an algebra $A$. Such deformations can be studied in terms of the cohomology of $(R,S)$.

\subsection{Associative Yang-Baxter pairs}\label{subsec-aybp}
Let $A$ be an associative algebra. In \cite{brz} the author introduced associative Yang-Baxter pairs as generalization of associative Yang-Baxter solutions \cite{aguiar}. Let us introduce the following notations first. For $r = r_{(1)} \otimes r_{(2)}$ and $s= s_{(1)} \otimes s_{(2)} \in A^{\otimes 2}$, we define
\begin{align*}
r_{13} s_{12} =  r_{(1)} s_{(1)} \otimes s_{(2)} \otimes r_{(2)}, \quad
r_{12} s_{23} =  r_{(1)} \otimes r_{(2)}  s_{(1)} \otimes s_{(2)}, \quad
r_{23} s_{13} =  s_{(1)} \otimes r_{(1)} \otimes r_{(2)}  s_{(2)}
\end{align*}
be elements of $A^{\otimes 3}$. A pair $(r,s )$ of elements of $A^{\otimes 2}$ is called an {\bf associative Yang-Baxter pair} if they satisfy
\begin{align*}
\begin{cases}
r_{13} r_{12} - r_{12} r_{23} + s_{23} r_{13} = 0,\\
s_{13} r_{12} - s_{12} s_{23} + s_{23} s_{13} = 0.
\end{cases}
\end{align*}

In \cite{brz} Brzezi\'{n}ski showed that if $(r= r_{(1)} \otimes r_{(2)}, s= s_{(1)} \otimes s_{(2)})$ is an associative Yang-Baxter pair, then the pair $(R,S)$ of linear maps on $A$ defined by 
\begin{align}\label{rbs-brz}
R(a) = r_{(1)}a r_{(2)} ~~ \text{ and } ~~ S(a) = s_{(1)} a s_{(2)}
\end{align}
is a Rota-Baxter system. In the following, we associate a generalized Rota-Baxter system to any skew-symmetric associative Yang-Baxter pair.

For $r = r_{(1)} \otimes r_{(2)}$ and $s = s_{(1)} \otimes s_{(2)}$, we define linear maps $r^\sharp , s^\sharp : A^* \rightarrow A$ by $r^\sharp (\alpha) =  \alpha (r_{(2)} ) r_{(1)}$ and $s^\sharp (\alpha) = \alpha (s_{(2)}) s_{(1)}$, for $\alpha \in A^*.$ With these notations, we have the following.

\begin{prop}
Let $r, s$ be two skew-symmetric elements of $A^{\otimes 2}$. Then $(r,s)$ is an associative Yang-Baxter pair if and only if $(r^\sharp, s^\sharp)$ is a generalized Rota-Baxter system on the coadjoint module $A^*$ over the algebra $A$.
\end{prop}

\begin{proof}
For $\alpha, \beta \in A^*$, we have $r^\sharp (\alpha) r^\sharp (\beta) = \alpha (r_{(2)}) \beta (\tilde{r}_{(2)}) r_{(1)} \tilde{r}_{(1)}$. On the other hand,
\begin{align*}
r^\sharp ( r^\sharp (\alpha) \cdot \beta + \alpha \cdot s^\sharp (\beta)) =~& r^\sharp \big(  \alpha(r_{(2)}) (r_{(1)} \cdot \beta) - \beta (s_{(1)}) (\alpha \cdot s_{(2)})  \big) \\
=~& \alpha(r_{(2)}) ((r_{(1)} \cdot \beta) (\tilde{r}_{(2)})) \tilde{r}_{(1)} - \beta ( s_{(1)}) ((\alpha \cdot s_{(2)})(r_{(2)})) r_{(1)} \\
=~& \alpha(r_{(2)})  \beta ( \tilde{r}_{(2)} r_{(1)}   ) \tilde{r}_{(1)} - \beta ( s_{(1)}) \alpha ( s_{(2)} r_{(2)}) r_{(1)}.
\end{align*}
We consider an auxiliary map $(\beta \otimes \alpha)^\sharp : A^{\otimes 3} \rightarrow A$ by $(\beta \otimes \alpha)^\sharp ( a \otimes b \otimes c) = \beta(b) \alpha(c) a$, for $a \otimes b \otimes c \in A^{\otimes 3}$. Note that $(r_{13} r_{12} - r_{12} r_{23} + s_{23} r_{13} ) = 0$ holds if and only if $(\beta \otimes \alpha)^\sharp (r_{13} r_{12} - r_{12} r_{23} + s_{23} r_{13} ) = 0$, for all $\alpha, \beta \in A^*$, or equivalently,
\begin{align*}
\beta (\tilde{r}_{(2)}) \alpha ( r_{(2)}) r_{(1)} \tilde{r}_{(1)} - \beta ( \tilde{r}_{(2)} r_{(1)} ) \alpha (r_{(2)}) \tilde{r}_{(1)} + \beta (s_{(1)}) \alpha (s_{(2)} r_{(2)}) r_{(1)} = 0.
\end{align*}
This is equivalent to 
\begin{align*}
r^\sharp (\alpha) r^\sharp (\beta) =  r^\sharp ( r^\sharp (\alpha) \cdot \beta + \alpha \cdot s^\sharp (\beta)), \text{ for all } \alpha , \beta \in A^*.
\end{align*}
Similarly, we have $s^\sharp (\alpha) s^\sharp (\beta) = \alpha (s_{(1)}) \beta (\tilde{s}_{(1)}) s_{(2)} \tilde{s}_{(2)}$ and
\begin{align*}
s^\sharp ( r^\sharp (\alpha) \cdot \beta + \alpha \cdot s^\sharp (\beta)) =~& s^\sharp \big( \alpha (r_{(2)}) (r_{(1)} \cdot \beta) - \beta (s_{(1)}) ( \alpha \cdot s_{(2)}) \big) \\
=~& - \alpha (r_{(2)} ) ((r_{(1)} \cdot \beta) (s_{(1)})) s_{(2)} + \beta (s_{(1)}) (( \alpha \cdot s_{(2)}) (\tilde{s}_{(1)})) \tilde{s}_{(2)} \\
=~& - \alpha (r_{(2)} ) \beta ( s_{(1)} r_{(1)} ) s_{(2)} + \beta (s_{(1)} ) \alpha ( s_{(2)} \tilde{s}_{(1)}) \tilde{s}_{(2)}.
\end{align*}
Here we consider the auxiliary map ${}^\sharp(\beta \otimes \alpha) : A^{\otimes 3} \rightarrow A$ defined by ${}^\sharp(\beta \otimes \alpha) ( a \otimes b \otimes c) =  \beta(a) \alpha(b) c$, for $a \otimes b \otimes c \in A^{\otimes 3}$. Then $s_{13} r_{12} - s_{12} s_{23} + s_{23} s_{13} = 0$ if and only if $ {}^\sharp(\beta \otimes \alpha) (s_{13} r_{12} - s_{12} s_{23} + s_{23} s_{13}) = 0$ , for all $\alpha, \beta \in A^*$. This is equivalent to
\begin{align*}
\beta (s_{(1)} r_{(1)}) \alpha (r_{(2)}) s_{(2)} - \beta (s_{(1)}) \alpha ( s_{(2)} \tilde{s}_{(1)}) \tilde{s}_{(2)} + \beta (s_{(1)} ) \alpha (\tilde{s}_{(1)}) \tilde{s}_{(2)} s_{(2)} = 0,
\end{align*}
or equivalently, $s^\sharp (\alpha) s^\sharp (\beta) = s^\sharp ( r^\sharp (\alpha) \cdot \beta + \alpha \cdot s^\sharp (\beta))$, for all $\alpha, \beta \in A^*$. Hence the proof.
\end{proof}

The above proposition shows that the dual space $A^*$ carries a dendriform structure given by
$\alpha \prec \beta = \beta (s_{(2)}) (\alpha \cdot s_{(1)})$ and $\alpha \succ \beta = \alpha (r_{(2)}) (r_{(1)} \cdot \beta).$
Hence the corresponding associative product on $A^*$ is given by 
\begin{align}\label{skew-aybp-ass}
\alpha * \beta =  r^\sharp (\alpha ) \cdot \beta + \alpha \cdot s^\sharp (\beta ) =  \alpha (r_{(2)}) (r_{(1)} \cdot \beta) + \beta (s_{(2)}) (\alpha \cdot s_{(1)}).
\end{align}
Moreover, the maps $r^\sharp , s^\sharp : A^* \rightarrow A$ are both associative algebra morphisms.

\medskip

Let $(r,s)$ and $(r', s')$ be two (skew-symmetric) associative Yang-Baxter pairs on $A$. They are said to be equivalent if there exists an algebra morphism $\phi : A \rightarrow A$ satisfying $(\phi \otimes \phi) (r) = r'$ and $(\phi \otimes \phi)(s) = s'$.

Next, we consider certain morphisms between skew-symmetric associative Yang-Baxter pairs that induce morphisms between corresponding generalized Rota-Baxter systems.

\begin{defn}
Let $(r,s)$ and $(r',s')$ be two skew-symmetric associative Yang-Baxter pairs. A {\bf weak morphism} from $(r,s)$ to $(r',s')$ is a triple $(\phi, \varphi, \psi)$ consist of two algebra morphisms $\phi, \varphi : A \rightarrow A$ and a linear map $\psi : A \rightarrow A$ satisfying
\begin{align*}
&(\psi \otimes \mathrm{id}_A) (r' ) = (\mathrm{id}_A \otimes \phi )(r) ~~~ \text{ and } ~~~ (\psi \otimes \mathrm{id}_A) (s' ) = (\mathrm{id}_A \otimes \varphi )(s) ,\\
& \psi ( a \phi (b)) = \psi (a) b ~~~ \text{ and } ~~~ \psi (\varphi (a) b ) = a \psi (b).
\end{align*}
\end{defn}

A weak morphism $(\phi, \varphi, \psi)$ is called a weak isomorphism if $\phi, \varphi, \psi$ are all linear isomorphisms. It follows that two skew-symmetric associative Yang-Baxter pairs $(r,s)$ and $(r', s')$ are equivalent if and only if there is an algebra morphism $\phi : A \rightarrow A$ such that $(\phi , \phi, \phi^{-1})$ is a weak morphism from $(r,s)$ to $(r',s').$

\begin{prop}
Let $(r,s)$ and $(r',s')$ be two skew-symmetric associative Yang-Baxter pairs. Then $(\phi, \varphi, \psi)$ is a weak (iso)morphism from $(r,s)$ to $(r',s')$ if and only if $(\phi, \varphi, \psi^*)$ is a (iso)morphism of generalized Rota-Baxter systems from $(r^\sharp, s^\sharp)$ to $(r'^\sharp, s'^\sharp)$.
\end{prop}

\begin{proof}
Take $r = r_{(1)} \otimes r_{(2)}$, $s = s_{(1)} \otimes s_{(2)}$, $r' = r'_{(1)} \otimes r_{(2)}'$ and $s' = s'_{(1)} \otimes s'_{(2)}$. Suppose $(\phi, \varphi, \psi^*)$ is a morphism of generalized Rota-Baxter systems from $(r^\sharp , s^\sharp)$ to $(r'^\sharp, s'^\sharp)$. Then by definition,
\begin{align}\label{grbs-mor-eq}
r'^\sharp \circ \psi^* (\xi) =\phi \circ r^\sharp (\xi), ~~~~ s'^\sharp \circ \psi^* (\xi ) = \varphi \circ s^\sharp (\xi), ~~~~ \psi^* ( a \cdot \xi) = \phi(a) \cdot \psi^* (\xi), ~~~~ \psi^* (\xi \cdot a) = \psi^* (\xi) \cdot \varphi (a),
\end{align} 
for all $a \in A$, $\xi \in A^*$. Note that $r^\sharp (\xi) = - \xi (r_{(1)}) r_{(2)}$ and $r'^\sharp (\xi) = - \xi ( r'_{(1)}) r'_{(2)}$. Therefore, for any $\xi, \eta \in A^*$, 
\begin{align*}
\langle (\psi \otimes \mathrm{id}_A ) (r'), \xi \otimes \eta \rangle = \langle \psi ( r'_{(1)}) \otimes r'_{(2)} , \xi \otimes \eta \rangle = \langle  \psi ( r'_{(1)}), \xi \rangle \langle r'_{(2)} , \eta \rangle 
= \langle r'_{(1)} , \psi^*\xi \rangle \langle r'_{(2)} , \eta \rangle 
= - \langle r'^\sharp (\psi^* \xi), \eta \rangle
\end{align*}
and
\begin{align*}
\langle (\mathrm{id}_A \otimes \phi )(r) , \xi \otimes \eta \rangle = \langle r_{(1)} \otimes \phi (r_{(2)}), \xi \otimes \eta \rangle = \langle r_{(1)} , \xi \rangle \langle \phi (r_{(2)}), \eta \rangle 
= \langle \phi (\xi (r_{(1)}) r_{(2)}), \eta \rangle = - \langle \phi (r^\sharp \xi ), \eta \rangle.
\end{align*}
Thus it follows from the first condition of (\ref{grbs-mor-eq}) that $(\psi \otimes \mathrm{id}_A ) (r') = (\mathrm{id}_A \otimes \phi )(r)$. Similarly, we can prove that $(\psi \otimes \mathrm{id}_A)(s') = (\mathrm{id}_A \otimes \varphi )(s)$ holds. Next we observe that
\begin{align*}
\langle \psi ( a\phi (b)) , \xi \rangle  = \langle \psi^* (\xi),a \phi (b) \rangle = \langle \phi (b) \cdot \psi^* (\xi) , a \rangle = \langle \psi^* (b \cdot \xi ), a \rangle = \langle b \cdot \xi  , \psi (a) \rangle = \langle \xi, \psi (a)b \rangle.
\end{align*} 
Thus, we have $\psi (a \phi (b) ) =\psi( a) b$. Similarly, we get $\psi ( \varphi(a) b ) = a \psi (b)$. Therefore, $(\phi, \varphi, \psi)$ becomes a weak morphism from $(r,s)$ to $(r',s')$.

The converse part is similar and we omit the details.
\end{proof}

Let $(r,s)$ be an associative Yang-Baxter pair on an associative algebra $A$. A deformation of $(r,s)$ is a pair of formal sums $(r_t = \sum t^i r_i,~ s_t = \sum t^i s_i )$ of elements of $A^{\otimes 2}$ with $r_0 = r,~ s_0 = s$ and such that $(r_t, s_t)$ is an associative Yang-Baxter pair. If $(r,s)$ is a skew-symmetric pair, then we assume that each $r_i$ and $s_i$ are also skew-symmetric. In this case, we say that $(r_t, s_t)$ is a skew-symmetric deformation of $(r,s)$.

If $(r_t = \sum t^i r_i, s_t = \sum t^i s_i)$ is a deformation of $(r,s)$, then $(R_t = \sum t^i R_i,~ S_t = \sum t^i S_i)$ is a deformation of the Rota-Baxter system $(R,S)$ given in (\ref{rbs-brz}), where $R_i (a) = r_{i (1)} a r_{i (2)}$ and  $S_i (a) = s_{i (1)} a s_{i (2)}$. On the other hand, if $(r_t = \sum t^i r_i, s_t = \sum t^i s_i)$ is a skew-symmetric deformation of a skew-symmetric Yang-Baxter pair $(r,s)$, then $(r_t^\sharp = \sum t^i r_i^\sharp,~ s_t^\sharp = \sum t^i s_i^\sharp)$ is a deformation of the generalized Rota-Baxter system $(r^\sharp, s^\sharp).$

\subsection{Covariant bialgebras}
Covariant bialgebras were introduced by Brzezi\'{n}ski \cite{brz} as an extension of infinitesimal bialgebras. In this subsection, we study perturbations of the coproduct in a covariant bialgebra and associate a pre-Lie algebra structure to any compatible covariant bialgebra. Some remarks about deformations are also mentioned.

Let $A$ be an associative algebra. Then $A \otimes A$ can be given an $A$-bimodule structure by
$a \cdot ( b \otimes c) = ab \otimes c$ and $(b \otimes c) \cdot a = b \otimes ca.$
Let $\delta_1, \delta_2 : A \rightarrow A \otimes A$ be two derivations on $A$ with values in the $A$-bimodule $A \otimes A$, i.e.
\begin{align*}
\delta_i (ab) = a \cdot \delta_i (b) + \delta_i (a) \cdot b, ~ \text{ for } i = 1,2.
\end{align*}
Then a linear map $\triangle : A \rightarrow A \otimes A$ is said to be a covariant derivation with respect to $(\delta_1, \delta_2)$ if
\begin{align*}
\triangle (ab) =  a \cdot \delta_1 (b) + \triangle(a) \cdot b   =   a \cdot \triangle(b) + \delta_2(a) \cdot b.
\end{align*}
\begin{defn}
A {\bf covariant bialgebra} consists of a tuple $(A, \mu, \triangle, \delta_1, \delta_2)$ in which $(A, \mu)$ is an associative algebra, $(A, \triangle)$ is an coassociative coalgebra, the maps $\delta_1, \delta_2 : A \rightarrow A \otimes A$ are derivations on $A$ with values in the bimodule $A \otimes A$ such that $\triangle$ is a covariant derivation with respect to $(\delta_1, \delta_2)$. 
\end{defn}

The following result constructs a covariant bialgebra from an associative Yang-Baxter pair \cite{brz}
Let $(A, \mu)$ be an associative algebra and $(r,s)$ be an associative Yang-Baxter pair. Define maps $\overline{\triangle}, \delta_r, \delta_s : A \rightarrow A \otimes A$ by
\begin{align}\label{quasi-p}
\overline{\triangle} (a) = a \cdot r - s \cdot a, \qquad \delta_r (a) = a \cdot r -r \cdot a , \qquad \delta_s (a) = a \cdot s -s \cdot a.
\end{align}
Then $(A, \mu, \overline{\triangle}, \delta_r, \delta_s)$ is a covariant bialgebra, called quasitriangular covariant bialgebra. 

In \cite{perturb1,perturb2} Drinfel'd consider the perturbations of the quasi-Hopf algebra structure. In the same spirit, we study here perturbations of the coproduct in a covariant bialgebra. Before that, we introduce some additional notations. Note that, the tensor product $A^{\otimes 3}$ also carries an $A$-bimodule structure by
\begin{align*}
a \cdot (b \otimes c \otimes d )  = ab \otimes c \otimes d  ~~~ \text{ and } ~~~ (b \otimes c \otimes d) \cdot a = b \otimes c \otimes da.
\end{align*}
We will use this $A$-bimodule in the next theorem.
\begin{thm}
Let $(A, \mu, \triangle, \delta_1, \delta_2)$ be a covariant bialgebra and $(r,s)$ be two elements of $A^{\otimes 2}$. Define $\overline{\triangle}, \delta_r, \delta_s : A \rightarrow A \otimes A$ by (\ref{quasi-p}). Then $(A, \mu, \triangle +\overline{\triangle}, \delta_1 + \delta_r, \delta_2 + \delta_s)$ is a covariant bialgebra if and only if 
\begin{align}\label{perturb}
a \cdot  \big( (\mathrm{id} \otimes \triangle)^{-} (r) - ( r_{13}r_{12} - r_{12} r_{23} + s_{23} r_{13} )  \big) -~& \big(  (\mathrm{id} \otimes \triangle)^{-} (s) -  ( s_{13} r_{12} - s_{12} s_{23} + s_{23}s_{13})  \big) \cdot a \\
&= s_{23} \triangle(a)_{13} + \triangle(a)_{13} r_{12}, \nonumber
\end{align}
for all $a \in A$.
Here $(\mathrm{id} \otimes \triangle)^{-} (r) = (\mathrm{id} \otimes \triangle) (r) - (\triangle \otimes \mathrm{id})(r)$ and 
$(\mathrm{id} \otimes \triangle)^{-} (s) = (\mathrm{id} \otimes \triangle) (s) - (\triangle \otimes \mathrm{id})(s).$
\end{thm}

\begin{proof}
It is easy to see that $\delta_r$ and $\delta_s$ are derivations, so the sums $\delta_1 + \delta_r$ and $\delta_2 + \delta_s$. We now show that $\triangle + \overline{\triangle}$ is a covariant derivation with respect to $(\delta_1 + \delta_r , \delta_2 + \delta_s)$. We have
\begin{align*}
(\triangle + \overline{\triangle}) (ab) =~& \triangle (ab) + \overline{\triangle} (ab) \\
=~& a \cdot \delta_1 (b) + \triangle (a) \cdot b + a \cdot \delta_r (b) + \overline{\triangle} (a) \cdot b = a \cdot (\delta_1 + \delta_r )(b) + (\triangle + \overline{\triangle})(a) \cdot b.
\end{align*} 
Similarly, we can show that $(\triangle + \overline{\triangle})(ab) = a \cdot (\triangle + \overline{\triangle})(b) + (\delta_2 + \delta_s) (a) \cdot  b.$ Thus we proved our claim. We are now left with the coassociativity of $\triangle + \overline{\triangle}$. We show that $\triangle + \overline{\triangle}$ is coassociative if and only if the condition (\ref{perturb}) holds. Note that
\begin{align}
(\mathrm{id} \otimes (\triangle + \overline{\triangle})) \circ ( \triangle + \overline{\triangle}) =~& (\mathrm{id} \otimes \triangle ) \circ \triangle + (\mathrm{id} \otimes \triangle) \circ \overline{\triangle} + (\mathrm{id} \otimes \overline{\triangle}) \circ \triangle + (\mathrm{id} \otimes \overline{\triangle}) \circ \overline{\triangle}, \label{perturb-1}\\
((\triangle + \overline{\triangle}) \otimes \mathrm{id} ) \circ (\triangle + \overline{\triangle}) =~& (\triangle \otimes \mathrm{id} ) \circ \triangle + (\triangle \otimes \mathrm{id} ) \circ \overline{\triangle} + (\overline{\triangle} \otimes \mathrm{id} ) \circ \triangle + (\overline{\triangle} \otimes \mathrm{id} ) \circ \overline{\triangle}. \label{perturb-2}
\end{align}
We have $(\mathrm{id} \otimes \triangle) \circ \triangle = (\triangle \otimes \mathrm{id} ) \circ \triangle$ as $\triangle$ is coassociative. Therefore, it follows from (\ref{perturb-1}) and (\ref{perturb-2}) that $(\triangle + \overline{\triangle})$ is coassociative if and only if
\begin{align*}
(\mathrm{id} \otimes \triangle - \triangle \otimes \mathrm{id} ) \circ \overline{\triangle} + (\mathrm{id} \otimes \overline{\triangle} - \overline{\triangle} \otimes \mathrm{id} ) \circ \triangle = (\overline{\triangle} \otimes \mathrm{id} - \mathrm{id} \otimes \overline{\triangle}) \circ \overline{\triangle},
\end{align*}
or equivalently,
\begin{align*}
(\mathrm{id} \otimes \triangle)^{-} \circ \overline{\triangle}(a) ~+~& (\mathrm{id} \otimes \overline{\triangle})^{-} \circ \triangle (a) \\
=~& a \cdot ( r_{13}r_{12} - r_{12} r_{23} + s_{23} r_{13}) - ( s_{13} r_{12} - s_{12} s_{23} + s_{23}s_{13} ) \cdot a.
\end{align*}
Here the right-hand side follows from \cite[Proposition 3.15]{brz}.
For $r = r_{(1)} \otimes r_{(2)}$ and $s = s_{(1)} \otimes s_{(2)}$, we have
\begin{align}\label{perr1}
&(\mathrm{id} \otimes \triangle)^{-} \circ \overline{\triangle}(a)  \nonumber \\
&=(\mathrm{id} \otimes \triangle - \triangle \otimes \mathrm{id} ) \big( a r_{(1)} \otimes r_{(2)} -  s_{(1)} \otimes s_{(2)}  a \big) \nonumber \\
&= a   r_{(1)} \otimes \triangle (r_{(2)}) - s_{(1)} \otimes \triangle (s_{(2)}  a) - \triangle (a  r_{(1)}) \otimes r_{(2)} + \triangle (s_{(1)}) \otimes s_{(2)} a \nonumber \\
&= a  r_{(1)} \otimes  \triangle ( r_{(2)})  - s_{(1)} \otimes s_{(2)} \cdot \triangle (a) - s_{(1)} \otimes \triangle (s_{(2)})  \cdot  a  \nonumber \\
&~~~ - a  \cdot \triangle (r_{(1)}) \otimes r_{(2)} - \triangle (a) \cdot r_{(1)} \otimes r_{(2)} +  \triangle ( s_{(1)}) \otimes s_{(2)} a \nonumber \\
&= a \cdot (( \mathrm{id} \otimes \triangle)(r)) - s_{12} \triangle(a)_{23} - ((\mathrm{id} \otimes \triangle)(s)) \cdot a - a \cdot ((\triangle \otimes \mathrm{id})(r)) - \triangle(a)_{12} r_{23} + ((\triangle \otimes \mathrm{id} )(s)) \cdot a  \nonumber \\
&= a \cdot (\mathrm{id} \otimes \triangle)^{-} (r) - (\mathrm{id} \otimes \triangle)^{-} (s) \cdot a - s_{12} \triangle(a)_{23} - \triangle(a)_{12} r_{23}.
\end{align}
On the other hand, by letting $\triangle (a) = a_{(1)} \otimes a_{(2)}$, we get
\begin{align}\label{perr2}
&(\mathrm{id} \otimes \overline{\triangle})^{-} \circ \triangle (a) \nonumber \\
&= (\mathrm{id} \otimes \overline{\triangle} - \overline{\triangle} \otimes \mathrm{id} ) (a_{(1)} \otimes a_{(2)}) \nonumber \\
&=  a_{(1)} \otimes a_{(2)} r_{(1)} \otimes r_{(2)} - a_{(1)} \otimes s_{(1)} \otimes s_{(2)} a_{(2)} - a_{(1)} r_{(1)} \otimes r_{(2)} \otimes a_{(2)} + s_{(1)} \otimes s_{(2)} a_{(1)} \otimes a_{(2)}  \nonumber \\
&= \triangle(a)_{12} r_{23} - s_{23} \triangle(a)_{13} - \triangle (a)_{13} r_{12} + s_{12} \triangle(a)_{23}.
\end{align}
Therefore, we get from (\ref{perr1}) and (\ref{perr2}) that
\begin{align*}
(\mathrm{id} \otimes \triangle)^{-} \circ \overline{\triangle}(a) + (\mathrm{id} \otimes \overline{\triangle})^{-}\circ \triangle (a) = a \cdot ((\mathrm{id} \otimes \triangle)^{-} (r)) - ((\mathrm{id} \otimes \triangle)^{-} (s)) \cdot a - s_{23} \triangle(a)_{13} - \triangle(a)_{13} r_{12}.
\end{align*}
The right-hand side is same as $a \cdot (r_{13}r_{12} - r_{12} r_{23} + s_{23} r_{13}) - (s_{13} r_{12} - s_{12} s_{23} + s_{23}s_{13}) \cdot a$ (equivalently, $(\triangle + \overline{\triangle})$ is coassociative) if and only if (\ref{perturb}) holds. 
\end{proof}

In the following, we construct a pre-Lie algebra from a suitable covariant bialgebra generalizing a result of Aguiar \cite{aguiar-pre-lie}. We need the following definition: A covariant bialgebra $(A, \mu, \triangle, \delta_1, \delta_2)$ is said to be a {\bf compatible covariant bialgebra} if 
\begin{align*}
(\mathrm{id} \otimes \delta_1) \circ \triangle = (\delta_2 \otimes \mathrm{id}) \circ \triangle.
\end{align*}

Any infinitesimal bialgebra is by definition a compatible covariant bialgebra. Another example of a compatible covariant bialgebra can be given by the following. Let $(A, \mu)$ be an associative algebra and $a, b \in A$ satisfying $a^2 = b^2 =  ba = 0$. Then $(r = a \otimes a, s = b \otimes b)$ is an associative Yang-Baxter pair. The corresponding quasitriangular covariant bialgebra $(A, \mu, \triangle', \delta_r, \delta_s)$ is a compatible covariant bialgebra.

The following result relates to compatible covariant bialgebras and pre-Lie algebras.

\begin{prop}
Let $(A, \mu, \triangle, \delta_1, \delta_2)$ be a compatible covariant bialgebra. Then $(A, \diamond)$ is a pre-Lie algebra, where $a \diamond b = b_{(1)} a b_{(2)}$, for $a, b \in A$ with $\triangle (b) = b_{(1)} \otimes b_{(2)}$.
\end{prop}

\begin{proof}
We have $( a \diamond b) \diamond c = c_{(1)} (a \diamond b) c_{(2)} = c_{(1)}b_{(1)} a b_{(2)} c_{(2)}$. On the other hand $a \diamond ( b \diamond c ) = a \diamond ( c_{(1)} b c_{(2)})$. To expand it, we observe that
\begin{align*}
\triangle ( c_{(1)} b c_{(2)}) =~& c_{(1)} b \cdot \delta_1 (c_{(2)}) + \triangle ( c_{(1)} b) \cdot c_{(2)} \\
=~& c_{(1)} b \cdot \delta_1 (c_{(2)}) + ( c_{(1)} \cdot \triangle (b) + \delta_2 (c_{(1)})\cdot b) \cdot c_{(2)} \\
=~& c_{(1)} b \cdot ( c_{(2)(1)}^1 \otimes c_{(2)(2)}^1) + \big( c_{(1)} \cdot (b_{(1)} \otimes b_{(2)}) + (c_{(1)(1)}^2 \otimes c_{(1)(2)}^2) \cdot b \big) \cdot c_{(2)} \\
=~& c_{(1)} b c_{(2)(1)}^1 \otimes c_{(2)(2)}^1 + c_{(1)} b_{(1)} \otimes b_{(2)} c_{(2)} + c^2_{(1)(1)} \otimes c^2_{(1)(2)} b c_{(2)}. 
\end{align*}
Here we have used the notations $\delta_1 (a) = a_{(1)}^1 \otimes a_{(2)}^1$ and $\delta_2 (a) = a_{(1)}^2 \otimes a_{(2)}^2$. Hence, we have
\begin{align*}
a \diamond ( b \diamond c) = a \diamond ( c_{(1)} b c_{(2)}) =   c_{(1)} b c_{(2)(1)}^1 a c_{(2)(2)}^1 + c_{(1)} b_{(1)} a b_{(2)} c_{(2)} + c^2_{(1)(1)} a c^2_{(1)(2)} b c_{(2)}.
\end{align*}
Therefore, 
\begin{align}\label{pre-1}
(a \diamond b ) \diamond c - a \diamond ( b \diamond c) = - c_{(1)} b c_{(2)(1)}^1 a c_{(2)(2)}^1 - c^2_{(1)(1)} a c^2_{(1)(2)} b c_{(2)}.
\end{align}
By interchanging $a$ and $b$, we get
\begin{align}\label{pre-2}
(b \diamond a ) \diamond c - b \diamond ( a \diamond c) = - c_{(1)} a c_{(2)(1)}^1 b c_{(2)(2)}^1 - c^2_{(1)(1)} b c^2_{(1)(2)} a c_{(2)}.
\end{align}
Finally, since $(A, \mu, \triangle, \delta_1, \delta_2)$ is a compatible covariant algebra, we have
$(\mathrm{id} \otimes \delta_1) \circ \triangle (c) = (\delta_2 \otimes \mathrm{id}) \circ \triangle (c)$, or equivalently,
\begin{align*}
c_{(1)} \otimes c_{(2)(1)}^1 \otimes c_{(2)(2)}^1 = c_{(1)(1)}^2 \otimes c_{(1)(2)}^2 \otimes c_{(2)}.
\end{align*}
Thus, it follows from (\ref{pre-1}) and (\ref{pre-2}) that $(a \diamond b) \diamond c - a \diamond ( b \diamond c) = (b \diamond a) \diamond c - b \diamond ( a \diamond c) $ holds. Hence the proof.
\end{proof}

Let $A$ be an associative algebra and $(r,s)$ be a skew-symmetric associative Yang-Baxter pair. Consider the associative algebra structure on $A^*$ given in (\ref{skew-aybp-ass}) induced from the generalized Rota-Baxter system $(r^\sharp, s^\sharp)$. On the other hand, we can consider the quasitriangular covariant bialgebra $(A, \mu, \overline{\triangle}, \delta_r, \delta_s)$ given in (\ref{quasi-p}). Since $\overline{\triangle}: A \rightarrow A \otimes A, ~ \overline{\triangle} (a) = a \cdot r - s \cdot a$ is coassociative, its dual defines an associative multiplication on $A^*$. The comparison between the above two associative structures on $A^*$ is given by the following.

\begin{prop}
Let $(r,s)$ be a skew-symmetric associative Yang-Baxter pair on an associative algebra $A$. Then the associative structure on $A^*$ induced from the generalized Rota-Baxter system $(r^\sharp, s^\sharp)$ coincides with the dual of the coassociative coproduct $\overline{\triangle} : A \rightarrow A \otimes A, ~ \overline{\triangle}(a) = a \cdot r-s \cdot a$.
\end{prop}

\begin{proof}
Note that the dual of $\overline{\triangle}$ that defines an associative product on $A^*$ is given by the following composition
\begin{align*}
A^* \otimes A^* \xrightarrow{~~~~ \Xi~~~~} (A \otimes A)^* \xrightarrow{\overline{\triangle}^*} A^*.
\end{align*}
Here $\Xi : A^* \otimes A^* \rightarrow (A \otimes A)^*$ is the map $\Xi ( \alpha \otimes \beta) (a \otimes b) = \alpha ( b) \beta (a)$. (Note that our map $\Xi$ is different than the standard one which is $\Xi ( \alpha \otimes \beta) (a \otimes b) = \alpha ( a) \beta (b)$.) For any $a \in A$, we have
\begin{align*}
\langle \overline{\triangle}^* (\Xi (\alpha \otimes \beta)), a \rangle =  \langle \Xi (\alpha \otimes \beta), \overline{\triangle}(a) \rangle =~&  \langle \Xi (\alpha \otimes \beta), a r_{(1)} \otimes r_{(2)} - s_{(1)} \otimes s_{(2)} a \rangle \\
=~& \alpha (r_{(2)}) \beta (a r_{(1)} ) - \alpha (s_{(2)}a) \beta (s_{(1)}).
\end{align*}
On the other hand,
\begin{align*}
\langle r^\sharp (\alpha ) \cdot \beta + \alpha \cdot s^\sharp (\beta ) , a \rangle = \alpha (r_{(2)}) \langle r_{(1)} \cdot \beta, a \rangle - \beta (s_{(1)}) \langle \alpha \cdot s_{(2)}, a \rangle = \alpha (r_{(2)}) \beta ( a r_{(1)}) - \beta (s_{(1)}) \alpha (s_{(2)}a ).
\end{align*}
Thus it follows that $\overline{\triangle}^* (\Xi (\alpha \otimes \beta)) = r^\sharp (\alpha ) \cdot \beta + \alpha \cdot s^\sharp (\beta )$. Hence the proof.
\end{proof}

Next, we introduce a notion of weak morphism between covariant bialgebras and relate them with weak morphism between associative Yang-Baxter pairs.

Let $(A, \mu, \triangle, \delta_1, \delta_2)$ and $(A', \mu', \triangle', \delta_1', \delta_2')$ be two covariant bialgebras. A morphism between them is given by an algebra map $\phi : A \rightarrow A'$  that is also a coalgebra map satisfying additionally $(\phi \otimes \phi ) \circ \delta_1 = \delta_1' \circ \phi$ and $(\phi \otimes \phi ) \circ \delta_2 = \delta_2' \circ \phi$. It is called an isomorphism of covariant bialgebras if $\phi$ is a linear isomorphism.


\begin{defn} A {\bf weak morphism} of covariant bialgebras from $(A, \mu, \triangle, \delta_1, \delta_2)$ to $(A, \mu, \triangle', \delta_1', \delta_2')$ consists of a triple $(\phi, \varphi, \psi)$ in which $\phi, \varphi : A \rightarrow A$ are algebra maps, $\psi : (A, \triangle') \rightarrow (A, \triangle)$ is a coalgebra map satisfying $(\psi \otimes \psi ) \circ \delta_1' = \delta_1 \circ \psi $, $(\psi \otimes \psi ) \circ \delta_2' = \delta_2 \circ \psi$ and the followings $\psi (  a\phi (b) ) = \psi (a) b$, $\psi ( \varphi (a) b) = a \psi (b)$. 
\end{defn}

Let $(A, \mu, \triangle, \delta_1, \delta_2)$ and $(A, \mu, \triangle', \delta_1', \delta_2')$ be two covariant bialgebras. Then it is easy to see that $\phi : A \rightarrow A$ is an isomorphism of covariant bialgebras if and only if $(\phi, \phi, \phi^{-1})$ is a weak isomorphism of covariant bialgebras.


\begin{prop}
Let $(A, \mu)$ be an associative algebra and $(r,s), (r', s')$ be two skew-symmetric associative Yang-Baxter pairs. If $(\phi, \varphi, \psi)$ is a weak morphism (resp. weak isomorphism) from $(r,s)$ to $(r',s')$, then $(\phi, \varphi, \psi)$ is a weak morphism (resp. weak isomorphism) of covariant bialgebras from $(A, \mu, \overline{\triangle}, \delta_{r}, \delta_{s})$ to $(A, \mu, \overline{\triangle}', \delta_{r'}, \delta_{s'})$.
\end{prop}

\begin{proof}
To prove that $(\phi, \varphi, \psi)$ is a weak morphism of covariant bialgebras, it remains to show that $\psi : (A, \overline{\triangle}') \rightarrow (A, \overline{\triangle})$ is a coalgebra map and satisfying $(\psi \otimes \psi) \circ \delta_{r'} = \delta_{r} \circ \psi$, ~ $(\psi \otimes \psi) \circ \delta_{s'} = \delta_{s} \circ \psi$. Note that $\psi$ is a coalgebra map if and only if $\psi^*$ is an algebra map. For any $\alpha , \beta \in A^*$ and $a \in A$, we have
\begin{align*}
\langle  \psi^* (\alpha *_{(r,s)} \beta ), a \rangle 
=~& \langle {r^\sharp (\alpha)} \cdot \beta  + \alpha \cdot s^\sharp (\beta) , ~ \psi (a) \rangle \\
=~& \langle \beta, \psi (a) r^\sharp (\alpha ) \rangle  + \langle \alpha, s^\sharp (\beta) \psi (a) \rangle \\
=~& \langle \beta , \psi (a  \phi r^\sharp (\alpha)) \rangle + \langle \alpha, \psi (\varphi s^\sharp (\beta) a ) \rangle \\
=~& \langle \beta , \psi (a  r'^\sharp \psi^* (\alpha)) \rangle + \langle \alpha, \psi ( s'^\sharp \psi^* (\beta)  a ) \rangle \\
=~& \langle \psi^* \beta , a  r'^\sharp \psi^* (\alpha) \rangle + \langle \psi^* \alpha, s'^\sharp \psi^* (\beta) a  \rangle \\
=~& \langle {(r'^\sharp \psi^* \alpha)} \cdot (\psi^* \beta) , a \rangle + \langle  (\psi^* \alpha) \cdot {(s'^\sharp \psi^* \beta)} , a \rangle = \langle \psi^*(\alpha ) *_{(r',s')} \psi^* (\beta) , a \rangle.
\end{align*}
This proves that $\psi^* $ is an algebra map. Moreover, the condition $(\psi \otimes \psi) \circ \delta_{r'} = \delta_{r} \circ \psi$ is equivalent to $\psi^* ( r^\sharp (\alpha) \cdot \beta + \alpha \cdot r^\sharp (\beta)) = r'^\sharp \psi^* (\alpha) \cdot \psi^* (\beta) + \psi^*(\alpha ) \cdot r'^\sharp \psi^* (\beta)$, the proof of which is similar to the above calculation (replacing $s$ by $r$). Similarly, for the condition $(\psi \otimes \psi) \circ \delta_{s'} = \delta_{s} \circ \psi$. Hence the proof.
\end{proof}

\subsection{Generalized averaging systems}
In this subsection, we consider a notion of the generalized averaging system as a generalization of averaging operator \cite{pei-guo} in the presence of bimodules. We will see that they are particular cases of generalized Rota-Baxter systems.

Let $A$ be an associative algebra and $M$ be an $A$-bimodule.

\begin{defn}
A {\bf generalized left (resp. right) averaging system} consists of a pair $(R, S)$ of linear maps from $M$ to $A$ satisfying
\begin{align*}
\begin{cases}
R(u) R(v) = R( R(u) \cdot v ) \\
S(u) S(v) = S( R(u) \cdot v ),
\end{cases}
\qquad \bigg(\mathrm{rsep.} ~~~~
\begin{cases}
R(u) R(v) = R( u \cdot S(v) ) \\
S(u) S(v) = S( u \cdot S(v) ),
\end{cases}
\bigg)
~~~ \text{ for } u,v \in M.
\end{align*}
A generalized averaging system is a pair $(R, S)$ which is both a left averaging system and a right averaging system.
\end{defn}

\begin{defn} \cite{loday}
An {\bf associative dialgebra} is a vector space $D$ together with bilinear maps $\dashv, \vdash : D \otimes D \rightarrow D$ satisfying the following identities
\begin{align*}
& a \dashv (b \dashv c) = (a \dashv b) \dashv c = a \dashv (b \vdash c),\\
& (a \vdash b) \dashv c = a \vdash (b \dashv c),\\
& (a \dashv b) \vdash c = a \vdash (b \vdash c) = (a \vdash b ) \vdash c, ~~~ \text{ for all } a, b , c \in D.
\end{align*}
\end{defn}

\begin{prop}
Let $(R, S)$ be generalized averaging system on $M$ over the algebra $A$. Then $M$ carries an associative dialgebra structure with products
\begin{align*}
u \dashv v = u \cdot S(v) ~~~~ \text{ and } ~~~~ u \vdash v = R(u) \cdot v, ~ \text{ for } u,v \in M.
\end{align*}
\end{prop}

\begin{proof}
For any $u,v,w \in M$, we have
\begin{align*}
&u \dashv ( v \dashv w ) = u \cdot S ( v \cdot S(w)) = u \cdot ( S(v) S(w)) = \begin{cases} ( u \cdot S(v)) \cdot S(w) = (u \dashv v) \dashv w \\
u \cdot S (R(v) \cdot w) = u \dashv ( v \vdash w),
\end{cases}\\
&( u \vdash v) \dashv w = (R(u) \cdot v) \cdot S(w) = R(u) \cdot (v \cdot S(w)) = u \vdash (v \dashv w),\\
&(u \dashv v) \vdash w = R (u \cdot S(v)) \cdot w = (R(u) R(v)) \cdot w= \begin{cases} R(u ) \cdot ( R(v) \cdot w ) = u \vdash ( v \vdash w) \\
R ( R(u) \cdot v) \cdot w = (u \vdash v) \vdash w.
\end{cases}
\end{align*}
Hence the proof.
\end{proof}

\begin{prop}
Let $r = r_{(1)} \otimes r_{(2)}$, $s = s_{(1)} \otimes s_{(2)}$ be two elements of $A^{\otimes 2}$ such that $r_{13} r_{12} = r_{12} r_{23}$ and $s_{13} r_{12} = s_{12} s_{23}$. Then the pair $(R,S)$ of linear maps on $A$ defined by $R(a) = r_{(1)} a r_{(2)}$ and $S(a) = s_{(1)} a s_{(2)}$ is a left averaging system on $A$.
\end{prop}

\begin{proof}
From the hypothesis, we have
\begin{align*}
r_{(1)} \tilde{r}_{(1)} \otimes \tilde{r}_{(2)} \otimes r_{(2)} = r_{(1)} \otimes r_{(2)} \tilde{r}_{(1)} \otimes \tilde{r}_{(2)} \quad \text{ and } \quad
s_{(1)} r_{(1)} \otimes r_{(2)} \otimes s_{(2)} = s_{(1)} \otimes s_{(2)} \tilde{s}_{(1)} \otimes \tilde{s}_{(2)}.
\end{align*}
In the above two identities, replacing the first tensor product by $a$ and the second product by $b$, and using the definition of $R$ and $S$, we get respectively $R(R(a)b) = R(a) R(b)$ and $S(R(a) b) = S(a) S(b)$. Hence the proof.
\end{proof}

\begin{remark}
Similar to the proof of above proposition, if $r,s$ satisfies $r_{12} r_{23} = s_{23} r_{13}$ and $s_{12} s_{23} = s_{23} s_{13}$, then the above-defined $(R,S)$ is a right averaging system. Therefore, if $r,s$ satisfies 
\begin{align}\label{fs-system}
r_{13} r_{12} = r_{12}r_{23} = s_{23} r_{13} ~~~ \text{ and } ~~~ s_{13} r_{12} = s_{12} s_{23} = s_{23} s_{13},
\end{align} 
then $(R,S)$ is a Rota-Baxter system. We call the above system (\ref{fs-system}) of equations as the Frobenius-separability system in view of \cite{beidar}.
\end{remark}

In the rest of this subsection, we will consider the generalized left averaging system. Note that if $(M, l, r)$ is an $A$-bimodule then $(M, l, 0)$ and $(M, 0, r)$ are both $A$-bimodules. Then it follows that any generalized left (resp. right) averaging system is a generalized Rota-Baxter system on the $A$-bimodule $(M, l, 0)$ (resp. $(M, 0, r)$). Therefore, we get the following results.

\begin{thm}
Let $A$ be an associative algebra and $M$ be an $A$-bimodule. Then there is a graded Lie algebra structure on the graded vector space $\bigoplus_{n \geq 0} \mathrm{Hom}(M^{\otimes n }, A \oplus A)$ whose Maurer-Cartan elements are given by generalized left (resp. right) averaging systems. 
\end{thm}

The induced cohomology groups are the cohomology of the generalized left (resp. right) averaging system. One may define deformations of a generalized left (resp. right) averaging system which are governed by the above cohomology. 

\section{Generalized Rota-Baxter systems on $A_\infty$-algebras}\label{sec-6}

In this section, we introduce generalized Rota-Baxter systems on a bimodule over an $A_\infty$-algebra. They are a generalization of Rota-Baxter operators on $A_\infty$-algebras considered in \cite{das1}.  We show that a generalized Rota-Baxter system induces a $Dend_\infty$-algebra structure. We first recall the notion of $A_\infty$-algebras from \cite{sta}.

\begin{defn}
An {\bf $A_\infty$-algebra} consists of a graded vector space $\mathcal{A} = \oplus \mathcal{A}_i$ together with a collection of multilinear maps $\{ \mu_k : \mathcal{A}^{\otimes k } \rightarrow \mathcal{A} |~ \mathrm{deg} (\mu_k ) =  k-2\}_{k \geq 1 }$ satisfying the following set of identities: for any $n \geq 1$,
\begin{align}\label{a-inf-id}
\sum_{i+j = n+1} \sum_{\lambda =1}^{j} (-1)^{\lambda (i+1) + i (|a_1| + \cdots + |a_{\lambda -1 }|)} ~ \mu_{j} \big(  a_1, \ldots, a_{\lambda -1}, \mu_i ( a_{\lambda}, \ldots, a_{\lambda + i-1}),
 a_{\lambda + i}, \ldots, a_n   \big) = 0,
\end{align}
for $a_i \in A_{|a_i|}, 1 \leq i \leq n.$
\end{defn}

When $\mu_k = 0$ for $k \geq 3$, we get differential graded associative algebras. If further $\mu_1 = 0$, one obtains graded associative algebras. An $A_\infty$-algebra whose underlying graded vector space $\mathcal{A}$ is concentrated in degree $0$ is nothing but an associative algebra.

Let $(\mathcal{A}, \mu_k)$ be an $A_\infty$-algebra. An $A_\infty$-bimodule over it consists of a graded vector space $\mathcal{M} = \oplus \mathcal{M}_i$ together with a collection of multilinear maps $\eta_k : \bigoplus_{i+1+j=k} (\mathcal{A}^{\otimes i} \otimes \mathcal{M} \otimes \mathcal{A}^{\otimes j}) \rightarrow \mathcal{M}$, for $k \geq 1$ with $\mathrm{deg}(\eta_k) = k-2$. These multilinear maps are supposed to satisfy the identities (\ref{a-inf-id}) with exactly one of $a_1, \ldots, a_n$ is from $\mathcal{M}$ and the corresponding multilinear operation $\mu$ is replaced by $\eta$.

It follows that any $A_\infty$-algebra $(\mathcal{A}, \mu_k)$ is an $A_\infty$-bimodule over itself.

\begin{defn}
Let $(\mathcal{A}, \mu_k)$ be an $A_\infty$-algebra and $(\mathcal{M}, \eta_k)$ be an $A_\infty$-bimodule. A pair $(R, S)$ of degree $0$ linear maps from $\mathcal{M}$ to $\mathcal{A}$ is said to be a {\bf generalized Rota-Baxter system} if they satisfy
\begin{align}
\mu_k (R(u_1), \ldots, R(u_k)) =~& R \big(    \sum_{i=1}^k  \eta_k ( R(u_1), \ldots , R(u_{i-1}), u_i, S(u_{i+1}), \ldots, S(u_k) ) \big), \label{rota-sys-defn1}\\
 \mu_k (S(u_1), \ldots, S(u_k)) =~& S \big(    \sum_{i=1}^k  \eta_k ( R(u_1), \ldots , R(u_{i-1}), u_i, S(u_{i+1}), \ldots, S(u_k) ) \big), \label{rota-sys-defn2}
\end{align}
for each $k \geq 1$ and $u_1, \ldots, u_k \in \mathcal{M}$.
\end{defn}

When $R = S$, we call $R$ a generalized Rota-Baxter operator. On the other hand if the $A_\infty$-bimodule is taken to be $\mathcal{A}$ itself, we call $(R,S)$ a Rota-Baxter system on the $A_\infty$-algebra $(\mathcal{A}, \mu_k)$.  Any Rota-Baxter system on an associative algebra can be seen as a Rota-Baxter system on an $A_\infty$-algebra whose underlying graded vector space is concentrated in degree $0$.

It is known that an $A_\infty$-algebra structure on $\mathcal{A}$ is equivalent to a square-zero coderivation (of degree $-1$) on the cofree coassociative coalgebra $T^c (s \mathcal{A})$. It could be interesting to find the interpretation of a Rota-Baxter operator (more generally, a Rota-Baxter system) on an $A_\infty$-algebra in terms of the compatibility with the coderivation on $T^c (s \mathcal{A})$.

The above definition of a (generalized) Rota-Baxter system is more justified by the following example and Theorem \ref{thm-rbs-dend}.

\begin{exam}
Let $A$ be an associative algebra and $(R, S)$ be a Rota-Baxter system on $A$. We denote these data by the triple $(A, R, S)$. A bimodule over the triple $(A, R,S)$ consists of an $A$-bimodule $M$ together with linear maps $R_M , S_M : M \rightarrow M$ satisfying
\begin{align*}
\begin{cases}
R(a) \cdot R_M (u) = R_M ( R(a) \cdot u + a \cdot S_M (u)) \\
S(a) \cdot  S_M (u) = S_M ( R(a) \cdot u + a \cdot S_M (u)),
\end{cases}
\quad
\begin{cases}
R_M (u)  \cdot R (a) = R_M ( R_M (u) \cdot a + u \cdot  R(a))\\
S_M (u) \cdot S (a) = S_M ( R_M (u) \cdot  a + u \cdot R(a)),
\end{cases}
\end{align*}
for $a \in A, u \in M$.
Let $(M, R_M, S_M)$ and $(N, R_N, S_N)$ be two bimodules over the triple $(A, R, S)$. A morphism between them is given by an $A$-bimodule map $ d : M \rightarrow N$ satisfying $R_N \circ d = d \circ R_M$ and $S_N \circ d =  d \circ S_M$. In this case, we first define an $A_\infty$-algebra structure on the $2$-term complex $M \xrightarrow{\mu_1 = d} A \oplus N$ with structure maps given by
\begin{align*}
\mu_2 ((a, n_1), (b, n_2 )) =~& (ab, a \cdot n_2 + n_1 \cdot b), \\
\mu_2 ((a, n_1), u) =~& a \cdot u, \\
\mu_2 (u, (a, n_1)) =~& u \cdot a,\\
\mu_k =~& 0, ~\text{ for } k \geq 3.
\end{align*}
Then it can be checked that $\overline{R} = (( R \oplus R_N) , R_M)$ and $\overline{ S} = (( S \oplus S_N), S_M)$ constitute a Rota-Baxter system on the above $A_\infty$-algebra.
\end{exam}

Next, we recall the notion of $Dend_\infty$-algebras introduced in \cite{lod-val-book}. Here we will consider the equivalent definition given in \cite{das1}. First recall from \cite{das1} that there are certain maps $R_0 (m; \overbrace{1, \ldots, \underbrace{n}_{i\text{-th}}, \ldots, 1}^m) : C_{m+n-1} \rightarrow C_m$ and $R_i (m; \overbrace{1, \ldots, \underbrace{n}_{i\text{-th}}, \ldots, 1}^{m}) : C_{m+n-1} \rightarrow \mathbb{K}[C_n]$ which are given by the following table:

\medskip

\begin{center}
\begin{tabular}{|c|c|c|}
\hline
$[r] \in C_{m+n-1}$ & $R_0 (m;1, \ldots, n, \ldots, 1)([r])$ & $R_i (m; 1, \ldots, n, \ldots, 1)([r])$ \\ \hline \hline
$1 \leq r \leq i-1$ & $[r]$ & $[1]+\cdots+[n]$\\ \hline
$i \leq r \leq i+n-1$ & $[i]$  & $[r-i+1]$\\ \hline
$ i+n \leq r \leq m+n-1$ & $[r-n+1]$ & $[1]+ \cdots+ [n]$ \\ \hline
\end{tabular}
\end{center}

\medskip

\begin{defn}
A {\bf $Dend_\infty$-algebra} consists of a graded vector space $\mathcal{A} = \oplus \mathcal{A}_i$ together with a collection $\{ \mu_{k, [r]} : \mathcal{A}^{\otimes k} \rightarrow \mathcal{A} ~| ~[r] \in C_k, \mathrm{deg}(\mu_{k, [r]}) =k-2 \}_{k \geq 1}$ of multilinear maps satisfying the following identities: for each $n \geq 1$ and $[r] \in C_n$,
\begin{align*}
\sum_{i+j= n+1} \sum_{\lambda = 1}^j~(-1)^{\lambda (i+1) + i (|a_1|+ \cdots + |a_{\lambda -1}|) } ~\mu_{j, R_0(j;1, \ldots, i, \ldots, 1)[r]} \big( a_1, \ldots, a_{\lambda -1},
\mu_{i, R_\lambda (j;1, \ldots, i, \ldots, 1)[r]} (a_\lambda, \ldots, a_{\lambda + i -1} ),\\ a_{\lambda + i}, \ldots, a_n \big) = 0, 
\end{align*}
for $a_1, \ldots, a_n \in \mathcal{A}.$
\end{defn}

A $Dend_\infty$-algebra $(\mathcal{A}, \mu_{k, [r]})$ whose underlying graded vector space $\mathcal{A}$ is concentrated in degree $0$ is a dendriform algebra with $\prec ~= \mu_{2, [1]}$ and $\succ ~= \mu_{2, [2]}.$ It has been proved in \cite{lod-val-book,das1} that if $(\mathcal{A}, \mu_{k, [r]})$ is a $Dend_\infty$-algebra then $(\mathcal{A}, \mu_k)$ is an $A_\infty$-algebra, where $\mu_k = \mu_{k, [1]} + \cdots + \mu_{k, [k]}$, for $1 \leq k < \infty$. 

The following theorem is the homotopy version of Proposition \ref{RBS-dend}.

\begin{thm}\label{thm-rbs-dend}
Let $(\mathcal{A}, \mu_k)$ be an $A_\infty$-algebra and $(\mathcal{M}, \eta_k)$ be an $A_\infty$-bimodule. If $(R, S)$ is  a generalized Rota-Baxter system, then $(\mathcal{M}, \mu_{k , [r]} )$ is a $Dend_\infty$-algebra where
\begin{align*}
\mu_{k, [r]} ( u_1, \ldots, u_k ) =  \eta_k \big( R(u_1), \ldots , R(u_{r-1}), u_r, S(u_{r+1}), \ldots, S(u_k) \big), ~\text{ for } k \geq 1 \text{ and } 1 \leq r \leq k.
\end{align*}
\end{thm}

\begin{proof}
Since $(R,S)$ is a generalized Rota-Baxter system, it follows from (\ref{rota-sys-defn1}) and (\ref{rota-sys-defn2}) that
\begin{align*}
\mu_k (R(u_1), \ldots, R(u_k)) = R (\sum_{i=1}^k \mu_{k, [i]} (u_1, \ldots, u_k)) ~~ \text{ and } ~~
\mu_k (S(u_1), \ldots, S(u_k)) = S (\sum_{i=1}^k \mu_{k, [i]} (u_1, \ldots, u_k)).
\end{align*}
On the other hand, the $A_\infty$ condition on $(\mathcal{A}, \mu_k)$ implies that
\begin{align}\label{a-inf-identi}
\sum_{i+j = n+1}^{} \sum_{\lambda =1}^{j} \pm ~ \mu_{j} \big( a_1, \ldots,  a_{\lambda -1}, \mu_i ( a_{\lambda}, \ldots, a_{\lambda + i-1}),
 a_{\lambda + i}, \ldots, a_n   \big) = 0,
\end{align} 
for all $n \geq 1$ and $a_1, \ldots, a_n \in A$. The same identity holds if exactly one of $(a_1, \ldots, a_n)$ is from $\mathcal{M}$ and the corresponding $\mu$ is replaced by $\eta$. Consider the elements $\big( R(u_1), \ldots,R(u_{r-1}), u_r, S(u_{r+1}), \ldots, S(u_n) \big)$, for some fixed $1 \leq r \leq n$.
For any fixed $i, j$ and $\lambda$, if $r \leq \lambda -1$, then the term inside the summation look
 \begin{align*}
& \eta_{j} \big( R  u_1, \ldots, u_r , \ldots  S u_{\lambda -1}, \mu_i ( Su_{\lambda}, \ldots, S u_{\lambda + i-1}),
S u_{\lambda + i}, \ldots, S u_n   \big)\\
& = \eta_j ( R  u_1, \ldots, u_r , \ldots  S u_{\lambda -1} , S (\sum_{k=1}^i \mu_{i, [k]} (u_\lambda , \ldots, u_{\lambda +i -1}) ), Su_{\lambda+i}, \ldots, Su_n \big) \\
& = \mu_{j, [r]} (u_1, \ldots, u_{\lambda -1} , \sum_{k=1}^i \mu_{i, [k]} (u_\lambda , \ldots, u_{\lambda +i -1}) , u_{\lambda +i}, \ldots, u_n)\\
& = \mu_{j, R_0 (j; 1, \ldots, i, \ldots, 1)[r]} (u_1, \ldots, u_{\lambda -1} , \mu_{i, R_\lambda (j; 1, \ldots, i, \ldots, 1)[r]} (u_{\lambda}, \ldots, u_{\lambda+i-1}), u_{\lambda +i}, \ldots, u_n).
 \end{align*}
If $\lambda \leq r \leq \lambda +i -1$, we get
\begin{align*}
& \eta_j (Ru_1, \ldots, R u_{\lambda -1}, \eta_i (Ru_\lambda, \ldots, u_r, \ldots, S u_{\lambda -i +1}), S u_{\lambda + i}, \ldots, S u_n) \\
& = \eta_j (Ru_1, \ldots, R u_{\lambda -1}, \mu_{i, [r - (\lambda -1)]} (u_\lambda, \ldots, u_{\lambda + i-1}), S u_{\lambda + i}, \ldots, S u_n ) \\
& = \mu_{j, \lambda } (u_1, \ldots, u_{\lambda -1}, \mu_{i, [r - (\lambda -1)]} (u_\lambda, \ldots, u_{\lambda + i-1}),  u_{\lambda + i}, \ldots,  u_n ) \\
& =  \mu_{j, R_0 (j; 1, \ldots, i, \ldots, 1)[r]} (u_1, \ldots, u_{\lambda -1} , \mu_{i, R_\lambda (j; 1, \ldots, i, \ldots, 1)[r]} (u_{\lambda}, \ldots, u_{\lambda+i-1}), u_{\lambda +i}, \ldots, u_n).
 \end{align*}
 Similarly, if $\lambda + i \leq r \leq n$,
 \begin{align*}
& \eta_j (Ru_1, \ldots, R u_{\lambda -1}, \mu_i (Ru_{\lambda}, \ldots, R u_{\lambda + i -1 }), R u_{\lambda + i}, \ldots, u_r , \ldots,  Su_n) \\
& =  \mu_{j, R_0 (j; 1, \ldots, i, \ldots, 1)[r]} (u_1, \ldots, u_{\lambda -1} , \mu_{i, R_\lambda (j; 1, \ldots, i, \ldots, 1)[r]} (u_{\lambda}, \ldots, u_{\lambda+i-1}), u_{\lambda +i}, \ldots, u_n).
 \end{align*}
Substitute these identities in (\ref{a-inf-identi}), we get the identities of a $Dend_\infty$-algebra.
\end{proof}

\section{Commuting Rota-Baxter systems and quadri-algebras}\label{sec-7}
In this section, we consider Rota-Baxter systems on a dendriform algebra and commuting Rota-Baxter systems on an associative algebra. We show how these structures induce quadri-algebras introduced by Aguiar and Loday \cite{aguiar-loday}. The homotopy version of these results is also described.

\begin{defn}
Let $D = (D, \prec, \succ)$ be a dendriform algebra. A {\bf Rota-Baxter system} on $D$ consists of a pair $(R,S)$ of linear maps on $D$ satisfying
\begin{align}\label{rbs-dend-alg}
\begin{cases} R(a) \prec R(b) = R (R(a) \prec b + a \prec S(b)) \\
S(a) \prec S(b) = S ( R(a) \prec b + a \prec S(b)), \end{cases} \quad
\begin{cases}  R(a) \succ R(b) = R ( R(a) \succ b + a \succ S(b)) \\
S(a) \succ S(b) = S( R(a) \succ b + a \succ S(b)),
\end{cases}
\end{align}
for $a, b \in D.$
\end{defn}

\begin{defn} \cite{aguiar-loday}
A {\bf quadri-algebra} is a vector space $A$ together with four binary operations $\nwarrow$ (north-west), $\nearrow$ (north-east), $\swarrow$ (south-west) and $\searrow$ (south-east) satisfying the following $3$ identities
\begin{align*}
&(a \nwarrow b ) \nwarrow c = a \nwarrow (b * c),&& (a \nearrow b) \nwarrow c = a \nearrow ( b \sqsubset c),&&  ( a \wedge b) \nearrow c = a \nearrow ( b \sqsupset c),\\
& ( a \swarrow b ) \nwarrow c = a \swarrow ( b \wedge c),&& (a \searrow b) \nwarrow c = a \searrow ( b \nwarrow c),&& ( a \vee b) \nearrow c = a \searrow ( b \nearrow c),\\
& (a \sqsubset b) \swarrow c = a \swarrow ( b \vee c),&& ( a \sqsupset b) \swarrow c = a \searrow ( b \swarrow c),&& ( a * b) \searrow c = a \searrow ( b \searrow c),
\end{align*}
for $a, b, c \in A$. Here we use the following notations 
\begin{align*}
&a \sqsubset b = a \nwarrow b + a \swarrow b, \qquad  a \sqsupset b = a \nearrow b + a \searrow b,\\
&a \wedge b = a \nwarrow b + a \nearrow b,  ~ \qquad  a \vee b = a \swarrow b + a \searrow b,\\
& a * b = a \sqsubset b +  a \sqsupset b =  a \wedge b + a \vee b = a \nwarrow b + a \nearrow b + a \swarrow b + a \searrow b.  
\end{align*}
\end{defn}

It turns out that in a quadri-algebra as above, $(A, \sqsubset, \sqsupset)$ and $(A, \wedge, \vee)$ are both dendriform algebras. Thus, $(A, *)$ is an associative algebra. Quadri-algebras are studied from multiplicative operadic points of view in \cite{das3}.

\begin{prop}\label{dend-rbs-quad}
Let $(D, \prec, \succ)$ be a dendriform algebra and $(R,S)$ be a Rota-Baxter system on $D$. Then $(D, \nwarrow, \nearrow,\swarrow, \searrow)$ is a quadri-algebra where
\begin{align*}
a \nwarrow b = a \prec S(b),\qquad a \nearrow b = a \succ S(b), \qquad a \swarrow b = R(a) \prec b, \qquad  a \searrow b = R(a) \succ b.
\end{align*}
\end{prop}

\begin{proof}
We will verify the first $3$ identities of a quadri-algebra. The verification of remaining $6$ identities are similar. For the first identity, we have
\begin{align*}
(a \nwarrow b) \nwarrow c = ( a \prec S(b)) \prec S(c) =~& a \prec ( S(b) \prec S(c)) + a \prec ( S(b) \succ S(c)) \\
=~& a \prec  S ( R(b) \prec c + b \prec S(c)) + a \prec  S ( R(b) \succ c + b \succ S(c)) =  a \nwarrow (b  * c).
\end{align*}
Similarly,
\begin{align*}
( a \nearrow b) \nwarrow c = ( a \succ S(b)) \prec S(c) = a \succ ( S(b) \prec S(c)) = a \succ S ( R(b) \prec c + b \prec S(c)) = a \nearrow ( b \sqsubset c),
\end{align*}
\begin{align*}
( a \wedge b) \nearrow c = ( a \prec S(b) + a \succ S(b) \succ S(c) = a \succ ( S(b) \succ S(c)) =~& a \succ S ( R(b) \succ c + b \succ S(c))  \\
=~& a \nearrow ( b \sqsupset c). 
\end{align*}
\end{proof}

Two Rota-Baxter systems $(P, Q)$ and $(R,S)$ on an associative algebra $A$ are said to {\bf commute} if
\begin{align*}
P \circ R = R \circ P, \qquad P \circ S = S \circ P, \qquad Q \circ R = R \circ Q, \qquad Q \circ S  = S \circ Q.
\end{align*}

\begin{prop}\label{rbs-rbs-comm}
Let $(P,Q)$ and $(R,S)$ be two commuting Rota-Baxter systems on an associative algebra $A$. Then $(R, S)$ is a Rota-Baxter system on the dendriform algebra $(A, \prec, \succ)$ induced from $(P, Q)$, i.e. $ a \prec b = a Q(b)$ and $a \succ b = P(a) b$, for $a,b \in A$.
\end{prop}

\begin{proof}
We have
\begin{align*}
R(a) \prec R(b) = R(a) ( QR(b)) = R(a) (RQ(b)) = R (R(a) Q(b) + a SQ(b))= R (R(a) \prec b + a \prec S(b)),
\end{align*}
and
\begin{align*}
S(a) \prec S(b) = S(a) (QS(b)) = S(a) (SQ(b)) = S (R(a) Q(b) + a SQ(b)) = R( R(a) \prec b + a \prec S(b)).
\end{align*}
Here we have verified the first two identities of (\ref{rbs-dend-alg}). The other two identities are similar to verify.
\end{proof}

Combining Propositions \ref{dend-rbs-quad} and \ref{rbs-rbs-comm}, we get the following.
\begin{prop}
Let $(P, Q)$ and $(R,S)$ be two commuting Rota-Baxter systems on an associative algebra $A$. Then $(A, \nwarrow, \nearrow,\swarrow, \searrow)$ is a quadri-algebra, where
\begin{align*}
a \nwarrow b = a QS(b),\quad a \nearrow b = P(a) S(b), \quad a \swarrow b = R(a) Q(b), \quad a \searrow b = PR(a) b.
\end{align*}
\end{prop}

\begin{exam}
Let $A$ be an associative algebra and $(r = r_{(1)} \otimes r_{(2)}, s = s_{(1)} \otimes s_{(2)} )$ be an associative Yang-Baxter pair. Then it has been shown in \cite[Proposition 3.4]{brz} that the pair $(R,S)$ is a Rota-Baxter system on $A$, where $R,S : A \rightarrow A$ are given by
\begin{align*}
R(a ) = r_{(1)} a r_{(2)}  \quad \text{ and } \quad S(a) = s_{(1)} a s_{(2)}.
\end{align*}
Suppose $(p= p_{(1)} \otimes p_{(2)}, q = q_{(1)} \otimes q_{(2)})$ is another associative Yang-Baxter pair (with associated Rota-Baxter system $(P,Q)$) such that $p$ and $q$ both commute with $r$ and $s$ as elements of $A \otimes A^{\mathrm{op}}$. Then $(P,Q)$ and $(R,S)$ are commuting Rota-Baxter systems.
\end{exam}

\begin{exam}
Let $A$ be an associative algebra and $\sigma, \tau$ be two commuting algebra maps. Suppose $P$ is a $\sigma$-twisted Rota-Baxter operator and $R$ is a $\tau$-twisted Rota-Baxter operator (see Example \ref{tw-rota}) such that $P \circ R = R \circ P$, $P \circ \tau = \tau \circ P$ and $R \circ \sigma = \sigma \circ R$. Then $(P, Q= \sigma \circ P)$ and $(R, S= \tau \circ R)$ are commuting Rota-Baxter systems.
\end{exam}

The above results can be even extended in the homotopy context. Note that the notion of $Quad_\infty$-algebras (homotopy quadri-algebras) can be explicitly define using \cite{das3,das-saha} where various Loday-type algebras and homotopy Loday-type algebras are considered. Here we first recall the definition.

Let $Q_n = C_n \times C_n$, for $n \geq 1$, where $C_n$'s are given in Subsection \ref{subsec-dend-rel}. Therefore, $Q_n = \{ ([r] ,[s]) |~ 1 \leq r, s \leq n \}$. We define the structure maps $\overline{R}_0 (m; 1, \ldots, n, \ldots, 1) : Q_{m+n-1} \rightarrow Q_m$ and $\overline{R}_i (m; 1, \ldots, n, \ldots, 1) : Q_{m+n-1} \rightarrow \mathbb{K}[C_n] \times \mathbb{K}[C_n]$ as the product of the structure maps defined by the table in Section \ref{sec-6}, i.e.
\begin{align*}
\overline{R}_j (m; 1, \ldots, n, \ldots, 1) ([r],[s]) = \big( R_j (m;1, \ldots, n, \ldots, 1) ([r]), R_j (m;1, \ldots, n, \ldots, 1) ([s]) \big), \text{ for } j=0,i.
\end{align*}

\begin{defn}
A {\bf $Quad_\infty$-algebra} is a graded vector space $\mathcal{A} = \oplus \mathcal{A}_i$ together with a collection of multilinear maps $\{ \mu_{k, ([r],[s])} : \mathcal{A}^{\otimes k} \rightarrow \mathcal{A} |~ ([r],[s]) \in Q_k,~ \mathrm{deg}(\mu_{k, ([r],[s])}) = k-2 \}_{k \geq 1}$ satisfying the followings: for each $n \geq 1$ and $([r],[s]) \in Q_n$,
\begin{align} \label{quad-inf-iden}
\sum_{i+j= n+1} \sum_{\lambda = 1}^j~(-1)^{\lambda (i+1) + i (|a_1|+ \cdots + |a_{\lambda -1}|) } ~\mu_{j, \overline{R}_0(j;1, \ldots, i, \ldots, 1)([r],[s])} \big( a_1, \ldots, a_{\lambda -1},
\mu_{i, \overline{R}_\lambda (j;1, \ldots, i, \ldots, 1)([r],[s])} (a_\lambda, \ldots, a_{\lambda + i -1} ),\\ a_{\lambda + i}, \ldots, a_n \big) = 0. \nonumber
\end{align}
\end{defn}

The following result is the homotopy version of the fact that a quadri-algebra induce two dendriform structures on the underlying vector space.

\begin{prop}
Let $(\mathcal{A}, \mu_{k, ([r],[s])} )$ be a $Quad_\infty$-algebra. Then $(\mathcal{A}, \mu'_{k, [r]})$ and $(\mathcal{A}, \mu''_{k, [r]})$ are $Dend_\infty$-algebras, where
\begin{align*}
\mu'_{k, [r]} = \mu_{k, ([r],[1])} + \cdots + \mu_{k, ([r],[k])} ~~~ \text{ and } ~~~ \mu''_{k, [r]} = \mu_{k, ([1],[r])} + \cdots + \mu_{k, ([k],[r])}, \text{ for } 1 \leq k < \infty,~ 1 \leq r \leq k.
\end{align*}
\end{prop}

\begin{proof}
In the identity (\ref{quad-inf-iden}), for any fix $[r] \in C_n$, we substitute $[s] = [1], \ldots, [n]$ and add all these identities. Then we simply get the $Dend_\infty$-algebra identities for $(\mathcal{A}, \mu_{k, [r]} )$. Similarly, for any fix $[s] \in C_n$, by substituting $[r] =[1], \ldots, [n]$ and adding all these, we get the $Dend_\infty$-algebra identities for $(\mathcal{A}, \mu''_{k, [r]})$.
\end{proof}

It follows from the above proposition that if $(\mathcal{A}, \mu_{k, ([r],[s])} )$ is a $Quad_\infty$-algebra, then $(\mathcal{A}, \mu_k)$ is an $A_\infty$-algebra, where
\begin{align*}
\mu_k = \mu'_{k,[1]} + \cdots + \mu'_{k,[k]} = \mu''_{k,[1]} + \cdots + \mu''_{k, [k]} = \sum_{([r],[s]) \in Q_k} \mu_{k, ([r],[s])}, ~ \text{ for } 1 \leq k < \infty. 
\end{align*}

\begin{defn}
Let $(\mathcal{A}, \mu_{k, [r]})$ be a $Dend_\infty$-algebra. A pair $(R,S)$ of degree $0$ linear maps on $\mathcal{A}$ is said to be a {\bf Rota-Baxter system} if for $1 \leq k < \infty$ and $1 \leq r \leq k$,
\begin{align}
\mu_{k, [r]} (R(a_1), \ldots, R(a_k)) =~& R \big( \sum_{i=1}^k ~\mu_{k, [r]} (R(a_1), \ldots, R(a_{i-1}), a_i, S(a_{i+1}), \ldots, S(a_k))  \big), \label{dend-inf-rbs1}\\
\mu_{k, [r]} (S(a_1), \ldots, S(a_k)) =~& S \big( \sum_{i=1}^k ~\mu_{k, [r]} (R(a_1), \ldots, R(a_{i-1}), a_i, S(a_{i+1}), \ldots, S(a_k))  \big). \label{dend-inf-rbs2}
\end{align}
\end{defn}

\begin{prop}
Let $(R,S)$ be a Rota-Baxter system in a $Dend_\infty$-algebra $(\mathcal{A}, \mu_{k, [r]})$. Then $(\mathcal{A}, \mu_{k, ([r],[s])} )$ is a $Quad_\infty$-algebra, where
\begin{align*}
\mu_{k, ([r], [s])} (a_1, \ldots, a_k ) = \mu_{k, [r]} ( R(a_1), \ldots, R(a_{s-1}), a_s, S(a_{s+1}), \ldots, S(a_k)),~ \text{ for } 1 \leq k < \infty,~ 1 \leq r,s \leq k.
\end{align*}
\end{prop}

\begin{proof}
The proof is similar to the approach of Theorem \ref{thm-rbs-dend}, and hence we omit the details.
\end{proof}

Let $(\mathcal{A}, \mu_k)$ be an $A_\infty$-algebra. Two Rota-Baxter systems $(P,Q)$ and $(R,S)$ are said to be compatible if $P \circ R = R \circ P$, $P \circ S = S \circ P$, $Q \circ R = R \circ Q$ and $Q \circ S = S \circ Q$.

\begin{prop}
Let $(P,Q)$ and $(R,S)$ be two commuting Rota-Baxter systems on an $A_\infty$-algebra $(\mathcal{A}, \mu_k)$. Then $(R,S)$ is a Rota-Baxter system on the $Dend_\infty$-algebra $(\mathcal{A}, \mu_{k, [r]})$ induced from the Rota-Baxter system $(P,Q)$ by Theorem \ref{thm-rbs-dend}. Consequently, there is a $Quad_\infty$-algebra $(\mathcal{A}, \mu_{k, ([r],[s])})$.
\end{prop}

\begin{proof}
We have
\begin{align*}
&\mu_{k, [r]} (R(a_1), \ldots, R(a_k)) \\
&= \mu_k ( PR(a_1), \ldots, PR(a_{r-1}), Ra_r, QR(a_{r+1}), \ldots, QR (a_k)) \\
&= \mu_k ( RP(a_1), \ldots, RP(a_{r-1}), Ra_r, RQ(a_{r+1}), \ldots, RQ (a_k)) \\
&= R \bigg( \sum_{i = 1}^{r-1}  \mu_k \big( RP(a_1), \ldots, RP(a_{i-1}), Pa_i, SP (a_{i+1}), \ldots, SP (a_{r-1}), Sa_r, SQ (a_{r+1}), \ldots, SQ (a_k)   \big)   \\
&  + \mu_{k} \big(    RP (a_1), \ldots, RP (a_{r-1}), a_r, SQ (a_{r+1}), \ldots, SQ (a_k) \big)   \\
&  +  \sum_{i = r+1}^k \mu_{k} \big( RP (a_1), \ldots, RP(a_{r-1}), Ra_r, RQ (a_{r+1}), \ldots, RQ (a_{i-1}), Q a_i, SQ (a_{i+1}), \ldots, SQ (a_k)  \big)   \bigg) \\
&= R \bigg( \sum_{i = 1}^{r-1}  \mu_k \big( PR(a_1), \ldots, PR(a_{i-1}), Pa_i, PS (a_{i+1}), \ldots, PS (a_{r-1}), Sa_r, QS (a_{r+1}), \ldots, QS (a_k)   \big)   \\
&  + \mu_{k} \big(    PR (a_1), \ldots, PR (a_{r-1}), a_r, QS (a_{r+1}), \ldots, QS (a_k) \big)   \\
&  +  \sum_{i = r+1}^k \mu_{k} \big( PR (a_1), \ldots, PR(a_{r-1}), Ra_r, QR (a_{r+1}), \ldots, QR (a_{i-1}), Q a_i, QS (a_{i+1}), \ldots, QS (a_k)  \big)   \bigg) \\
& = R \big( \sum_{i=1}^k \mu_{k, [r]} (R(a_1), \ldots, R(a_{i-1}), a_i, S(a_{i+1}), \ldots, S (a_{k}) ) \big).
\end{align*}
Thus, we have proved the identity (\ref{dend-inf-rbs1}). The proof of (\ref{dend-inf-rbs2}) is similar.
\end{proof}

\medskip

\noindent {\em Acknowledgements.} The research is supported by the postdoctoral fellowship of Indian Institute of Technology (IIT) Kanpur. The author thanks the Institute for support. This work is completely done at home during the lockdown period for COVID-19. He also wishes to thank his family members for support.

\end{document}